\documentclass[a4paper,12pt]{article}
\usepackage[english]{babel}
\usepackage[utf8x]{inputenc}
\usepackage{amsmath,amsfonts,amssymb,amsthm}
\usepackage{mathtools}
\usepackage{hyperref}
\usepackage{enumitem}
\usepackage{float}
\usepackage{xcolor}
\usepackage{tikz-cd}
\usepackage{titlesec}
\usepackage{graphicx}
\usepackage[all,cmtip]{xy}

\def    \C      {{\mathbb C}}
\def    \R      {{\mathbb R}}
\def \Z {{\mathbb Z}}
\def \SS {{\mathbb S}}
\newcommand{\xx}{\mathbf{x}}
\newcommand{\yy}{\mathbf{y}}
\renewcommand{\epsilon}{\varepsilon}

\newtheorem{theorem}{Theorem}[section]

\newtheorem{lemma}[theorem]{Lemma}
\newtheorem{prop}[theorem]{Proposition}
\newtheorem{rmk}[theorem]{Remark}

\title{Symplectic embeddings into disk cotangent bundles}
\date{}
\author{Brayan Ferreira and Vinicius G. B. Ramos}

\AtEndDocument{\bigskip{\footnotesize
  \textsc{Instituto de Matem\'atica Pura e Aplicada, Rio de Janeiro, Brazil} \par  
  \textit{E-mails}:  \texttt{brayan.ferreira@impa.br, vgbramos@impa.br}
  }}

\begin{document}

\maketitle

\begin{abstract}
In this paper, we compute the embedded contact homology (ECH) capacities of the disk cotangent bundles $D^*S^2$ and $D^*\R P^2$. We also find sharp symplectic embeddings into these domains. In particular, we compute their Gromov widths. In order to do that, we explicitly calculate the ECH chain complexes of $S^*S^2$ and $S^* \R P^2$ using a direct limit argument on the action inspired by Bourgeois's Morse--Bott approach and ideas from Nelson--Weiler's work on the ECH of prequantization bundles. Moreover, we use integrable systems techniques to find explicit symplectic embeddings. In particular, we prove that the disk cotangent bundles of a hemisphere and of a punctured sphere are symplectomorphic to an open ball and a symplectic bidisk, respectively.
\end{abstract}

\section{Introduction}
The study of symplectic embeddings is central in symplectic topology. Much is known about embeddings from and into four-dimensional toric domains, see \cite{concaveconvex,schlenk}. In \cite{bidisk}, the second author studies symplectic embeddings from and into the lagrangian bidisk $D^2\times D^2$. In particular, it is shown that $D^2\times D^2$ is symplectomorphic to a toric domain which allows one to use the ECH machinery developped in \cite{hutchings2011quantitative,concaveconvex}.

An important example of a symplectic manifold is the cotangent bundle $T^*\Sigma$ of a surface $\Sigma$ equipped with the canonical symplectic form. After a choice of a Riemannian metric on $\Sigma$, one can look at the disk cotangent bundle $D^* \Sigma\subset T^* \Sigma$, which is the subset of the covectors $p\in T^*\Sigma$ such that $\Vert p\Vert\le 1$. The goal of this article is to study some symplectic embedding problems into $D^*S^2$ and $D^*\R P^2$, where $S^2$ and $\R P^2$ are equipped with the standard round Riemannian metrics. The main result is the calculation of the Gromov widths of these disk cotangent bundles.

If $(X_1,\omega_1)$ and $(X_2,\omega_2)$ are symplectic manifolds, we write $(X_1,\omega_1)\hookrightarrow (X_2,\omega_2)$ whenever there exists a symplectic embedding $\varphi:X_1\to X_2$ such that $\varphi^*\omega_2=\omega_1$. For $a>0$, let $B(a)$ denote the four-dimensional ball of capacity $a$, i.e., 
\[B(a)=\left\{z\in\R^4\mid \pi\Vert z\Vert^2\le a\right\},\]
endowed with the standard symplectic structure $\omega_0$.
Given a symplectic manifold $(X,\omega)$, its Gromov width $c_{Gr}(X,\omega)$ is defined to be the supremum of $a$ such that $(B(a),\omega_0)\hookrightarrow (X,\omega)$.
For $a,b>0$, the ellipsoid $E(a,b)$ and the symplectic polydisk $P(a,b)$ are defined as
\begin{equation}
\label{def:ebp}
\begin{aligned}
 E(a,b) &= \left\{(z_1,z_2) \in \C^2=\R^4 \mid \left(\frac{\pi|z_1|^2}{a} + \frac{\pi |z_2|^2}{b}\right) \le 1\right\},\\
 P(a,b)&=\left\{(z_1,z_2)\in\C^2=\R^4 \mid \pi|z_1|^2<a,\pi|z_2|^2<b\right\}.\\ 
\end{aligned}\end{equation}
We denote by $\mathrm{int} X$ the interior of a set $X\subset \R^4$. We now state our main result.
\begin{theorem}\label{thm:embedding}
Let $\omega_0$ be the standard symplectic form on $\R^4$ and let $\omega_{can}$ be the canonical symplectic form on a cotangent bundle. Then
\[c_{Gr}(D^* S^2,\omega_{can})=c_{Gr}(D^* \R P^2,\omega_{can})=2\pi.\] Moreover,
\begin{enumerate}[label=(\roman*)]
\item $(\mathrm{int}B(2\pi),\omega_0)\hookrightarrow (D^* S^2,\omega_{can})$,
\item $(\mathrm{int}B(2\pi),\omega_0)\hookrightarrow (D^* \R P^2,\omega_{can})$,
\item $(\mathrm{int}E(2\pi,4\pi),\omega_0)\hookrightarrow (D^* S^2,\omega_{can})$,
\item $(\mathrm{int}P(2\pi,2\pi),\omega_0)\hookrightarrow (D^* S^2,\omega_{can})$.
\end{enumerate}
\end{theorem}
\begin{rmk}
The symplectic embeddings in (ii), (iii) and (iv) above are volume filling. We also note that since there is no closed exact lagrangian submanifold in $\C^2$ \cite{gromov1985pseudo}, embeddings in the converse directions do not exist.
\end{rmk}

The proof of Theorem \ref{thm:embedding} has two parts. The first one is the computation of ECH capacities and the second one is the proof that full measure subsets of $D^* S^2$ and $D^* \R P^2$ are symplectomrphic to an open polydisk and an open ball, respectively. To explain the first part, recall that ECH capacities are a nondecreasing sequence of non-normalized symplectic capacities for four-dimensional symplectic manifolds. In other words, given a four-dimensional $(X,\omega)$, there exists a sequence
\[0=c_0(X,\omega)\le c_1(X,\omega)\le c_2(X,\omega)\le \dots\le \infty.\]
These satisfy the usual homogeneity, monotonicity and non-triviality conditions. Moreover, they have been computed for many toric domains in $\R^4$, see \cite{hutchings2011quantitative,concave}. They have been shown to be sharp for many symplectic embedding problems, e.g. for embedding concave into convex toric domains, see \cite{concaveconvex}.

We now state the main result that is used to prove Theorem \ref{thm:embedding}. Let $N(a,b)$ denote the sequence of all nonnegative integer linear combinations of $a$ and $b$, arranged in nondecreasing order, and indexed starting at $0$. Hutchings has shown that the sequence of ECH capacities of $E(a,b)$ agrees with the sequence $N(a,b)$ in \cite{hutchings2011quantitative}. If $\mathcal{S}$ is a sequence of integers and $j\in\Z$, we denote by $M_j(\mathcal{S})$ the subsequence of $\mathcal{S}$ of the multiples of $j$, in the order that they appear.
\begin{theorem}\label{thm:capacities}
The ECH capacities of $D^*S^2$ and $D^* \R P^2$ are given by
\begin{enumerate}[label= (\alph*)]
\item $(c_k(D^*S^2,\omega_{can}))_k  = 2\pi M_2(N(1,1)) = (0,4\pi,4\pi,4\pi,8\pi,8\pi,8\pi,8\pi,8\pi,\ldots).$ \label{parta}
\item $(c_k(D^*\R P^2,\omega_{can}))_k = \pi M_4(N(1,1))= (0,4\pi,4\pi,4\pi,4\pi,4\pi,8\pi,\ldots)$. \label{partb}
\end{enumerate}
\end{theorem}

The second part of the proof of Theorem \ref{thm:embedding} is the following result.
\begin{theorem}\label{thm:symplecto}
Let $q\in S^2$. Then $(D^*(S^2\setminus\{q\},\omega_{can})$ is symplectomorphic to $(\mathrm{int}P(2\pi,2\pi),\omega_0)$. Moreover for any open hemisphere $\Sigma\subset S^2$, $(D^*\Sigma,\omega_{can})$ is symplectomorphic to $(\mathrm{int}B(2\pi),\omega_0)$.
\end{theorem}

Using Theorems \ref{thm:capacities} and \ref{thm:symplecto}, we can now prove Theorem \ref{thm:embedding}.

\begin{proof}[Proof of Theorem \ref{thm:embedding}]
We first observe that (i), (ii) and (iv) follow directly from Theorem \ref{thm:symplecto}. Moreover, it follows from  work of Frenkel--M\"{u}ller \cite{frenkelmuller} that \[(\mathrm{int}E(2\pi,4\pi),\omega_0)\hookrightarrow (\mathrm{int} P(2\pi,2\pi),\omega_0).\] Hence (iii) also holds. 

Since the embedding $(\mathrm{int} B(2\pi),\omega_0)\hookrightarrow (D^*\R P^2,\omega_{can})$ is volume filling, it follows that $c_{Gr}(D^*\R P^2,\omega_{can})=2\pi$. From Theorem \ref{thm:capacities}(a) if $(B(a),\omega_0)\hookrightarrow (D^*S^2,\omega_{can})$, then \[2a=c_3(B(a),\omega_0)\le c_3(D^*S^2,\omega_{can})=4\pi.\]
So $a\le 2\pi$. Therefore $c_{Gr}(D^*S^2,\omega_{can})=2\pi$. 
\end{proof}

\begin{rmk}
Felix Schlenk pointed out to us that the Gromov width of $D^* S^2$ can be alternatively computed by combining results from \cite{oakleyusher} and \cite{lagbarriers}.
\end{rmk}

\noindent{\bf Structure of the paper:} In Section 2, we recall the definition of ECH and introduce our strategy to prove Theorem \ref{thm:capacities}. In Section 3, we prove Theorem \ref{thm:capacities}(a) and in Section 4, we prove Theorem \ref{thm:capacities}(b). Lastly, in Section 5, we use integrable systems to prove Theorem \ref{thm:symplecto}.

\vskip 8pt

\noindent{{\bf Acknowledgments:}}
We would like to thank Jo\'{e} Brendel, Jo Nelson and Felix Schlenk for helpful conversations. The second author is partially supported by a grant from the Serrapilheira Institute, the FAPERJ grant Jovem Cientista do Nosso Estado and the CNPq grants~407510/2018-4 and~306405/2020-2.

\section{Foundations of embedded contact homology}
Let $Y$ be a closed $3$-manifold equipped with a contact form $\lambda$, i.e., $\lambda$ is an $1$-form such that $\lambda \wedge d\lambda>0$, and let $\xi = \ker \lambda$ be the contact structure. The Reeb vector field $R_\lambda$ is the unique vector field in $Y$ satisfying:
$$i_{R_\lambda}d\lambda = 0 \quad \text{and} \quad \lambda(R_\lambda) = 1.$$
Let $\varphi_t$ the Reeb flow, i.e., the flow associated to $R_\lambda$. The closed trajectories of $\varphi_t$ are called Reeb orbits. A Reeb orbit $\gamma \colon \R/T\Z \to Y$ is nondegenerate when the linearized return map $P_\gamma := d\varphi_T|_\xi \colon \xi_{\gamma(0)} \to \xi_{\gamma(0)}$ does not admit $1$ as an eigenvalue. The contact form $\lambda$ is nondegenerate if all Reeb orbits are nondegenerate. Suppose that $\lambda$ is nondegenerate. Since $P_\gamma$ is a linear symplectomorphism, it turns out that the Reeb vector field admits three types of closed orbits:
\begin{enumerate}
\item \emph{Elliptic}: orbits $\gamma$ such that the eigenvalues of the linearized return map $P_\gamma$ are norm one complex numbers.
\item \emph{Positive hyperbolic}: when eigenvalues of $P_\gamma$ are positive real numbers.
\item \emph{Negative hyperbolic}: when eigenvalues of $P_\gamma$ are negative real numbers.
\end{enumerate}
An orbit set $\alpha = \{(\alpha_i,m_i)\}$ is a finite set, where $\alpha_i$ are distinct embedded Reeb orbits on $Y$ and $m_i$ are positive integers. An admissible orbit set is an orbit set such that $m_i = 1$ whenever $\alpha_i$ is hyperbolic. We denote the homology class of an orbit set $\alpha$ by
$$[\alpha] = \sum_i m_i [\alpha_i] \in H_1(Y).$$
For a fixed $\Gamma \in H_1(Y)$, and a generic symplectization-admissible almost complex structure $J$ on $\R \times Y$, the chain complex $ECC_*(Y,\lambda, \Gamma, J)$ is the $\Z_2$-vector space generated by the admissible orbit sets in homology class $\Gamma$, and its differential counts certain $J$-holomorphic curves in $\R \times Y$, as explained below. This chain complex gives rise to the embedded contact homology $ECH_*(Y,\lambda, \Gamma, J)$. Taubes proved in \cite{taubes2010embedded} that the latter is isomorphic to the \emph{from} version of Seiberg-Witten Floer cohomology $\widehat{HM}^{-*}(Y,\mathfrak{s}_\xi + PD(\Gamma))$. In particular, $ECH_*(Y,\lambda, \Gamma, J)$ does not depend on $\lambda$ or $J$, and so we write $ECH_*(Y,\xi,\Gamma)$. There is a refinement of ECH: the filtered ECH, whose definition we quickly recall. \\
The symplectic action of an orbit set $\alpha = \{(\alpha_i,m_i)\}.$ is defined by
$$\mathcal{A}(\alpha) = \sum_i m_i \int_{\alpha_i} \lambda.$$
Since the almost complex structure is admissible on the symplectization, the restriction of $d\lambda$ to any $J$-holomorphic curve in $\R \times Y$ is pointwise nonnegative. Therefore, Stokes' theorem yields that the differential decreases the symplectic action, and so there is a subcomplex $ECC^L(Y,\lambda,\Gamma,J)$ which is the span of chain complex generators $\alpha$ such that $\mathcal{A}(\alpha)<L$. The homology of this subcomplex is the filtered ECH, denoted by $ECH^L(Y,\lambda,\Gamma)$. We summarize some concepts and properties of ECH ingredients in next three sections. A nice introduction to this theory can be found in \cite{hutchings2014lecture}.

\subsection{The ECH index}
Let $\mathcal{T}(\gamma)$ be the set of homotopy classes of symplectic trivializations of $\xi|_\gamma$. This is an affine space over $\Z$: given two trivializations $\tau_1, \tau_2 \colon \xi|_\gamma \to S^1 \times \R^2$, we denote by $\tau_1 - \tau_2$ the degree of $\tau_1 \circ \tau_2^{-1}\colon S^1 \to Sp(2,\R) \equiv S^1$. Let $\alpha=\{(\alpha_i,m_i)\}, \beta=\{(\beta_j,n_j)\}$ be two orbit sets. If $\tau \in \mathcal{T}(\alpha,\beta)\colon= \Pi_i \mathcal{T}(\alpha_i) \times \Pi_j \mathcal{T}(\beta_j)$, the elements of $\mathcal{T}(\alpha_i)$ and $\mathcal{T}(\beta_j)$ are denoted by $\tau_i^+$ and $\tau_j^-$.
\subsubsection{Conley-Zehnder index} Let $\gamma \colon \R/T\Z \to Y$ be a parametrized Reeb orbit, and $\tau \in \mathcal{T}(\gamma)$. If $\varphi_t$ is the Reeb flow, the derivative
$$d\varphi_t \colon T_{\gamma(0)}Y \to T_{\gamma(t)}Y$$
restricts to a linear symplectomorphism $\psi_t \colon \xi_{\gamma(0)} \to \xi_{\gamma(t)}$. Using the trivialization $\tau$, the latter can be viewed as a $2 \times 2$ symplectic matrix for each $t$. Since $\lambda$ is nondegenerate, this gives rise to a path of symplectic matrices starting at the identity $I_{2 \times 2}$ and ending at the linearized return map $\psi_T = P_\gamma$, which does not have $1$ as an eigenvalue. So the Conley-Zehnder index $CZ_\tau(\gamma) \in \Z$ is defined as the Conley-Zehnder index of the path $\{\psi_t\}_{t \in [0,T]}$. For the definition of this index, see e.g. \cite[\S 2.4]{salamon1999lectures}. It can be explicitly given as follows.\

If $\gamma$ is hyperbolic, the matrices $\psi_t$ rotate the eigenvectors by angle $\pi k$ for some integer $k$ (which is even when $\gamma$ is positive hyperbolic and odd in the negative hyperbolic case), then $CZ_\tau(\gamma) = k.$ When $\gamma$ is elliptic, up to a homotopy, we can assume that $\psi_t$ is a rotation by $2\pi\theta_t$, where $\theta_0 = 0$ and $\theta_t$ is a continuous function of $t \in [0,T]$. Then we call $\theta_T$ by rotation angle of $\gamma$ with respect to $\tau$ and $CZ_\tau(\gamma) = 2\lfloor \theta_T \rfloor + 1$. In particular, if a simple orbit and all its interations have odd Conley-Zehnder index, then it must be elliptic.\

If one changes the trivialization $\tau$, the Conley-Zehnder index changes in the following way:
\begin{equation}\label{czmuda}
CZ_\tau(\gamma^k) - CZ_{\tau^\prime}(\gamma^k) = 2k(\tau - \tau^\prime).
\end{equation}

\subsubsection{Relative first Chern class} We denote by $H_2(Y,\alpha,\beta)$ the affine space over $H_2(Y)$ that consists of $2$-chains $\Sigma$ in $Y$ with
$$\partial \Sigma = \sum_i m_i \alpha_i - \sum_j n_j \beta_j$$
modulo boundaries of $3$-chains. Let $Z \in H_2(Y,\alpha,\beta)$ and $\tau \in \mathcal{T}(\alpha,\beta)$. Given a surface $S$ with boundary and a smooth map $f\colon S \to Y$ representing $Z$, the relative first Chern class $c_\tau(Z) = c_1(\xi|_{f(S)},\tau) \in \Z$ is defined as the signed count of zeros of a generic section $\mathfrak{s}$ of $f^*\xi$ obtained by extending a nonvanishing section of $f^*\xi_{\partial S}$ with is trivial with respect to $\tau$.\

Let $\alpha^\prime$ and $\beta^\prime$ be other two orbit sets. The number $c_\tau$ is linear in the relative homology class, i.e.,
\begin{equation}\label{ctaulinear}
c_\tau(Z+Z^\prime) = c_\tau(Z) + c_\tau(Z^\prime).
\end{equation}
Moreover, if we change the trivialization $\tau$, then
\begin{equation}\label{ctaumuda}
c_\tau(Z) - c_{\tau^\prime}(Z) = \sum_i m_i ({\tau^\prime_i}^+ - \tau_i^+) - \sum_j n_j({\tau^\prime_j}^- - \tau_j^-).
\end{equation}

\subsubsection{Relative intersection number}

Let $\pi_Y\colon \R \times Y \to Y$ denotes the projection and take a smooth map $f \colon S \to [-1,1] \times Y$, where $S$ is a compact oriented surface with boundary, such that $f|_{\partial S}$ consists of positively oriented covers of $\{1\} \times \alpha_i$ with multiplicity $m_i$ and negatively oriented covers of $\{-1\} \times \beta_j$ with multiplicity $n_j$, $\pi_Y \circ f$ represents $Z$, the restriction $f|_{\dot{S}}$ to the interior of $S$ is an embedding, and $f$ is transverse to $\{-1,1\} \times Y$. Such an $f$ is called an admissible representative for $Z \in H_2(Y,\alpha,\beta)$ and we abuse notation denoting this representative as $S$. Furthermore, suppose that $\pi_Y|S$ is an immersion near $\partial S$ and $S$ contains $m_i$ (resp. $n_j$) singly covered circles at $\{1\} \times \alpha_i$ (resp. $\{-1\} \times \beta_j$) such that the $m_i$ (resp. $n_j$) nonvanishing sections of $\xi$ over $\alpha_i$ (resp. $\beta_j$), given by projecting conormal vectors in $S$, are $\tau$-trivial. Moreover, in each fiber of $\xi$ over $\alpha_i$ or $\beta_j$, these sections lie in distinct rays. Then $S$ is a $\tau$-representative. \

Let $\tau \in \mathcal{T}(\alpha \cup \alpha^\prime, \beta \cup \beta^\prime)$. Moreover, let $S$ and $S^\prime$ be $\tau$-representatives of two classes, $Z \in H_2(Y,\alpha,\beta),\ Z^\prime \in H_2(Y,\alpha^\prime,\beta^\prime)$, respectively, such that the projected conormal vectors at the boundary all lie in different rays. Then $Q_\tau(Z,Z^\prime) \in \Z$ is the signed count of (transverse) intersections of $S$ and $S^\prime$ in $(-1,1) \times Y$. We denote write $Q_\tau(Z) := Q_\tau(Z,Z)$ and we note that $Q_\tau$ is quadratic in the following sense
\begin{equation}\label{qtauquadratic}
Q_\tau(Z+Z^\prime) = Q_\tau(Z) + 2Q_\tau(Z,Z^\prime) + Q_\tau(Z^\prime).
\end{equation}
Finally, if $Z,Z^\prime \in H_2(Y,\alpha,\beta)$, changing the trivialization yields
\begin{equation}\label{qtaumuda}
Q_\tau(Z,Z^\prime) - Q_{\tau^\prime}(Z,Z^\prime) = \sum_i m_i^2({\tau_i^\prime}^+ - \tau_i^+) - \sum_j n_j^2 ({\tau_j^\prime}^- - \tau_j^-).
\end{equation}
For these and more about the numbers $CZ_\tau, c_\tau$ and $Q_\tau$, see e.g. \cite{hutchings2002index,hutchings2009embedded}.

\subsubsection{Linking numbers} Given two nullhomologous oriented knots $K,K^\prime \subset Y$, the linking number between them is defined as the intersection number
$$lk(K,K^\prime) = K \cdot S_{K^\prime} \in \Z,$$
where $S_{K^\prime}$ is an embedded Seifert surface for $K^\prime$ transverse to $K$. If $K$ is a transverse knot, i.e., it can be realized as an embedding $\gamma\colon S^1 \to Y$ such that $\gamma^\prime(t) \notin \xi_{\gamma(t)}$ for all $t \in S^1$, there is defined the self-linking number
$$sl(K, S_K) = lk(K, K^\prime) \in \Z,$$
where $K^\prime$ is a parallel copy of $K$ obtained by pushing $K$ in the direction of a non-vanishing section of $\xi|_{S_K}$ and $S_K$ is an embedded Seifert surface for $K$. When the Euler class $e(\xi) \in H^2(Y)$ vanishes, the self-linking number does not depend on $S_K$ and, hence, we write $sl(K)$. We recommend \cite[Chapter 3]{geiges2008introduction} for further details on these and others knots invariants.
\

Let $\alpha = \{(\alpha_i,m_i)\}, \beta = \{(\beta_j,n_j)\}$ be two orbit sets in the homology class $\Gamma$ and $Z \in H_2(Y, \alpha, \beta)$. The ECH index is defined by
\begin{equation}\label{echindex}
I(\alpha, \beta; Z) = c_\tau(Z) + Q_\tau(Z) + CZ^I_\tau(\alpha) - CZ^I_\tau(\beta),
\end{equation}
where $CZ^I_\tau(\alpha) = \sum_i \sum_{k=1}^{m_i} CZ_\tau(\alpha_i^k)$ and similarly for $CZ^I_\tau(\beta)$.
This index has the following properties.

\begin{prop}~\cite[Proposition 1.6]{hutchings2002index} \label{ECHind}
The ECH index satisfies:
\begin{enumerate}[label=\alph*)]
\item (Well defined) $I(\alpha,\beta, Z)$ does not depend on $\tau$ (although each term in the formula does).
\item (Additivity) $I(\alpha,\beta, Z+W) = I(\alpha,\delta, Z) + I(\delta,\beta, W)$, whenever $\delta$ is another orbit set in $\Gamma$, $Z \in H_2(Y,\alpha,\delta)$ and $W \in H_2(Y,\delta,\beta)$. \label{additivity}
\item (Index parity) If $\alpha$ and $\beta$ are chain complex generators, then
$$(-1)^{I(Z)} = \varepsilon(\alpha)\varepsilon(\beta),$$
where $\varepsilon(\alpha)$ denotes $(-1)$ to the number of positive hyperbolic orbits in $\alpha$ and similarly does $\varepsilon(\beta)$.
\item (Index ambiguity Formula) $I(\alpha,\beta, Z) - I(\alpha,\beta, Z^\prime) = \langle c_1(\xi)+2PD(\Gamma), Z-Z^\prime \rangle$, where $c_1(\xi)$ is the first chern class of the vector bundle $\xi$ and $PD$ denotes the Poincaré dual.
\end{enumerate}
\end{prop}

\subsection{Differential and grading}\label{grading}
The other important index is the Fredholm index which describes, for a generic almost complex structure $J$, the dimension of the moduli space near a $J$-holomorphic curve $C$, see~\cite[Proposition 3.1]{hutchings2014lecture}. It is defined as
\begin{equation}\label{fredholm}
\text{ind}(C) = -\chi(C) + 2c_\tau(C) + \sum_{i=1}^k CZ_\tau(\gamma_i^+) - \sum_{i=1}^l CZ_\tau(\gamma_i^-),
\end{equation}
where $\chi(C)$ denotes the Euler characteristic of the $J$-holomorphic curve $C$ with $k$ positive ends at Reeb orbits $\gamma_1^+,\ldots, \gamma_k^+$ and $l$ negative ends at Reeb orbits $\gamma_1^-,\ldots, \gamma_l^-$.
\

Furthermore, for a generic almost complex structure $J$, we can understand the behavior of low ECH index $J$-holomorphic currents. Here a trivial cylinder means a cylinder $\R \times \gamma$, where $\gamma$ is a Reeb orbit.

\begin{prop}~\cite[Proposition 3.7]{hutchings2014lecture} \label{lowind}
Suppose $J$ is generic. Let $\alpha$ and $\beta$ be orbit sets and let $\mathcal{C} \in \mathcal{M}(\alpha,\beta)$ be any $J$-holomorphic current in $\R \times Y$, not necessarily somewhere injective. Then
\begin{enumerate}[label=\roman*)]
\item $I(\mathcal{C}) \geq 0$, with equality if and only if $\mathcal{C}$ is a union of trivial cylinders with multiplicities.
\item If $I(\mathcal{C}) = 1$, then $\mathcal{C} = \mathcal{C}_0 \sqcup C_1$, where $I(\mathcal{C}_0) = 0$, and $C_1$ is embedded and has $\text{ind}(C_1) = I(C_1) = 1$.
\item If $I(\mathcal{C}) = 2$, and if $\alpha$ and $\beta$ are chain complex generators, then $\mathcal{C} = \mathcal{C}_0 \sqcup C_2$, where $I(\mathcal{C}_0) = 0$, and $C_2$ is embedded and has $\text{ind}(C_2) = I(C_2) = 2$.
\end{enumerate}
\end{prop}
Given two chain complex generators $\alpha$ and $\beta$, the chain complex differential $\partial$ coefficient $\langle \partial \alpha, \beta \rangle \in \Z_2$ is a$\mod 2$ count of ECH index $1$ $J$-holomorphic curves in the symplectization of $Y$ that \emph{converge as currents} to $\sum_i m_i \alpha_i$ as $s \to \infty$ and to $\sum_j n_j \beta_j$ as $s \to -\infty$, see e.g. \cite{hutchings2014lecture}.
\

It follows from Proposition \ref{ECHind} that $I$ gives rise to a relative $\Z_d$ grading on the chain complex $ECC_*(Y,\lambda,\Gamma,J)$, where $d$ is the divisibility of $c_1(\xi) + 2PD(\Gamma)$ in $H^2(Y;\Z) \mod$ torsion. In order to define an (non-canonical) absolute $\Z_d$ grading, it is enough to fix some generator $\beta \in \Gamma$ and set
$$|\alpha| = I(\alpha,\beta):= [I(\alpha,\beta,Z)],$$
for an arbitrary $Z \in H_2(Y,\alpha,\beta)$. By additivity property \ref{additivity} in Proposition \ref{ECHind}, the differential decreases this absolute grading by $1$. Moreover, when $c_1(\xi) + 2PD(\Gamma)$ is torsion in $H^2(Y;\Z)$, we obtain a $\Z$ grading on $ECC_*(Y,\lambda,\Gamma,J)$.
\subsection{U map and ECH capacities} \label{Umap}
There is a map
$$U \colon ECH_*(Y,\xi,\Gamma) \to ECH_{*-2}(Y,\xi,\Gamma),$$
when $Y$ is connected, induced by a chain map which counts ECH index $2$ curves passing through a base point $(0,z)$ in $\R \times Y$, see \cite[\S 3.8]{hutchings2014lecture}. Hutchings used this map to define a sequence
$$c_0(Y,\lambda) \leq c_1(Y,\lambda) \leq c_2(Y,\lambda) \leq c_3(Y,\lambda) \leq \cdots \leq \infty.$$
It is defined as $c_0(Y,\lambda) = 0$ and
$$c_k(Y, \lambda) = \inf_L \left\{\exists \ \eta \in ECH^L(Y,\lambda,0)| \ U^k\eta = [\emptyset]\right\}, \quad \text{for} \quad k\geq 1.$$
This sequence is the ECH spectrum of $(Y,\lambda)$. When $\lambda$ is a degenerate contact form, one defines
\begin{equation}\label{cklim}
c_k(Y,\lambda) = \lim_{n\to \infty} c_k(Y,f_n\lambda),
\end{equation}
where $f_n \colon Y \to \R_{>0}$ are functions on $Y$, with $f_n{\lambda}$ nondegenerate contact forms and $\lim_{n\to \infty} f_n = 1$ in the $C^0$ topology. Now suppose that $(X,\omega)$ is a symplectic filling of $(Y,\lambda)$, meaning that $X$ is a four dimensional manifold with boundary $\partial X = Y$ and $\omega|_Y = d\lambda$. 
It turns out that for each $L>0$ there are homomorphisms
$$\Phi^L(X,\omega)\colon ECH^L(Y,\lambda,0) \to \Z/2,$$
such that $\Phi^L([\emptyset]) = 1$, in particular, this ensures that $[\emptyset] \neq 0 \in ECH(Y,\lambda,0)$.
For the symplectic four dimensional manifold $(X,\omega)$, its ECH capacities are defined as
$$c_k(X,\omega) = c_k(Y,\lambda) \in [0,\infty],$$
and these are in fact symplectic invariants. For a collection of properties of the numbers $c_k(X,\omega)$ including an unexpected property relating the asymptotics of ECH capacities to the volume of the manifold $(X,\omega)$, see Theorem 1.3 in \cite{hutchings2014lecture}.

\begin{rmk}\label{rmk:umap}
Taubes proved in \cite{taubes2010embeddedV} that the $U$ map on ECH agrees with the analogous $U$ map on Seiberg-Witten Floer cohomology. In particular, it is a topological invariant of $Y$.
\end{rmk}

\subsection{Morse-Bott approach to Zoll contact forms}
A contact form $\lambda$ on $Y$ is Morse-Bott if the action spectrum
$$\sigma(\lambda) = \{\mathcal{A}(\gamma)| \ \gamma \ \text{Reeb orbit of} \ \lambda\}$$
is discrete and if, given $T \in \sigma(\lambda)$ we always have $N_T=\{p \in Y; \ \varphi_T(p) = p\}$, where $\varphi$ denotes the Reeb flow, as a closed submanifold of $Y$ such that the rank of $d\lambda|_{N_T}$ is locally constant and $T_pN_T = \ker(d(\varphi_T) - I)_p$. In particular, if $\lambda$ is a Zoll contact form, i.e., every Reeb orbit is closed and has the same minimal period, $\lambda$ is Morse-Bott and such that $N_T = Y$ for the common minimal period $T$ of the orbits; suppose that $\lambda$ is Zoll. The idea then is to find a function $\overline{f} \colon Y \to \R$ such that $d\overline{f}(R_\lambda) = 0$ related with a Morse function $f \colon \SS \to \R$, where $\SS = Y/S^1$ is the orbit space, and doing a small perturbation on the contact form without changing the contact structure. It is done in such a way that, when one fixes an action range, the Reeb orbits from the new one are nondegenerate and correspond to the critical points of $f$. Indeed, the vanishing condition along the Reeb direction ensures that $\overline{f}$ descends to the quotient, and $f$ is this induced function.  We note that since $\lambda$ is Zoll, the action on $Y$ given by the Reeb flow is a free $S^1$-action, and thus the orbit space $\SS$ is a compact surface. Let $\mathfrak{q}\colon Y \to \SS$ be the quotient map of the orbit space.
\begin{lemma}\label{orbitspace}
Let $\lambda$ be a Zoll contact form on a closed 3 dimensional manifold $Y$. Then the Reeb orbits space $\SS = Y/S^1$ is diffeomorphic to an orientable closed surface.
\begin{proof}
The closed surface $\SS$ does admit a symplectic form. In fact, since $\ker d\lambda = \langle R_\lambda \rangle$ and the Reeb vector field $R_\lambda$ is the infinitesimal generator of the action, $d\lambda$ is basic, i.e., there exists a unique $2$-form $\omega_\lambda$ defined on $\SS$ such that $\mathfrak{q}^*\omega_\lambda = d\lambda$. Moreover, the splitting $TY = \langle R_\lambda \rangle \oplus \xi$ and the differential of quotient map $d\mathfrak{q}$ ensure that $T\SS \cong \xi$, and thus the fact that $d\lambda|_\xi$ is nondegenerate yields the nondegeneracy of $\omega_\lambda$.
\end{proof}
\end{lemma}

Now we choose a Morse function $f \colon \SS \to \R$ and define $\overline{f} = \mathfrak{q}^*f = f \circ \mathfrak{q} \colon Y \to \R$. For small positive $\varepsilon$, we define the perturbed contact form
\begin{equation} \label{lambdaepsilon}
\lambda_\varepsilon = (1+\varepsilon \mathfrak{q}^*f)\lambda.
\end{equation}
Using the splitting $TY = \langle R_\lambda \rangle \oplus \xi$, where $R_\lambda$ is the Reeb vector field defined by $\lambda$ and $\xi = \ker \lambda$, one can understand better the periodic orbits of the new Reeb vector field with fixed action range.

\begin{lemma}{\cite[Lemma 2.3]{bourgeois2002morse}} \label{bourg}
For each $T$, there exists $\varepsilon=\varepsilon(T)>0$ small enough such that the periodic orbits of $R_{\lambda_\varepsilon}$ in $Y$ of action $T^\prime \leq T$ are nondegenerate and correspond to the critical points of $f$.
\end{lemma}
Given a critical point $p\in \SS$ of $f$, and denoting by $\gamma_p^k$, $k = 1,2,\ldots$, the corresponding Reeb orbit and its iterations, the construction of Bourgeois gives the Conley-Zehnder indices of these orbits, for small $\varepsilon>0$ as in the previous lemma. We recall here that using a trivialization $\tau$ of $\xi$ along a simple Reeb orbit $\gamma$, we obtain a path of symplectic matrices $\Psi_\gamma(t)$, which has a Robin-Salamon index $\mu_{RS}(\Psi_\gamma(t))$ as defined in $\cite{robbin1993maslov}$. This index can be seen as a generalized Conley-Zenhder index when the path of symplectic matrices does not necessarily end at a non-degenerate matrix. One then sets $\mu^\tau_{RS}(\gamma):=\mu_{RS}(\Psi_\gamma(t))$. Using this index, Bourgeois computed the Conley-Zehnder index of a Reeb orbit of a perturbed contact form.
\begin{lemma}\cite[ Lemma 2.4]{bourgeois2002morse} 
Let $T>0$, $\varepsilon>0$ as in Lemma \ref{bourg}, and let $k$ be a positive integer such that the action $kT^\prime$ of $\gamma_p^k$ is less than or equal to $T$. Therefore the Conley-Zehnder index of $\gamma_p^k$ as a Reeb orbit of $\lambda_\varepsilon$ is given by
\begin{equation}
CZ_\tau(\gamma_p^k) = \mu^\tau(\gamma_p^k) - \frac{1}{2} \dim \SS + \text{ind}_f(p),
\end{equation}
where $\mu^\tau(\gamma_p^k)$ denotes the Robin-Salamon index of $\gamma_p^k$ as a Reeb orbit of $\lambda$.
\end{lemma}

In particular, if the orbit space $\SS$ is homeomorphic to the sphere $S^2$, we take a Morse function $f\colon \SS \to \R$ with exactly two critical points $p_1,p_2 \in \SS$ of Morse indices $0$ and $2$, respectively. This way, we get the Conley-Zehnder indices:
\begin{equation}\label{CZ}
CZ_\tau(\gamma_{p_1}^k) =  \mu^\tau(\gamma_{p_1}^k) -1 \quad \text{and} \quad CZ_\tau(\gamma_{p_2}^k) =\mu^\tau(\gamma_{p_2}^k) + 1.
\end{equation}

\subsection{Example: three dimensional sphere}\label{s3} \label{esferas3}
Let $Y=\partial B(1)\subset \R^4 \cong \C^2$ and let
$$\lambda_0 =  \frac{1}{2}\sum_{k=1}^2 \left(x_k dy_k - y_k dx_k\right)$$ the standard Liouville form in $\C^2$. Note that $\lambda_0$ is indeed a contact form on $Y\cong S^3$.
It turns out that the Reeb orbits of $\lambda_0$ are the circles of the Hopf fibration. In particular, $\lambda_0$ is Zoll and the orbit space is the sphere $S^2$. Furthermore, considering a trivialization $\tau$ given by $T\C^2 = \xi_0 \oplus \xi_0^\omega$, where $\xi_0^\omega$ denotes the symplectic complement of $\xi_0$, and using an identification between the symplectization $\R \times Y$ and $\C^2 \backslash \{0\}$, one concludes that $\mu^\tau(\gamma^k) = 4k$, see \cite[Lemma 4.6]{nelson2020automatic}, when $\gamma$ is a simple Reeb orbit. Then \eqref{CZ} yields
\begin{equation}\label{CZs3}
CZ_\tau(\gamma_{p_1}^k) =  4k -1 \quad \text{and} \quad CZ_\tau(\gamma_{p_2}^k) = 4k + 1,
\end{equation}
and thus, since these are odd numbers, the Reeb orbits that generates the filtered chain complex $ECC^T_*(Y,\lambda_{\varepsilon(T)},0,J)$ are all elliptic. By the index parity property in Proposition \ref{ECHind}, the ECH index will be always even, and this implies that the differential
$$\partial\colon ECC^T_*(Y,\lambda_{\varepsilon(T)},0,J) \to ECC^T_*(Y,\lambda_{\varepsilon(T)},0,J)$$
vanishes for any almost complex structure $J$. As an immediate consequence, the filtered ECH homology $ECH^T_*(Y,\lambda_{\varepsilon(T)},0,J)$ is isomorphic to its chain complex. The period of any simple Reeb orbit before the perturbation is $1$. Therefore, $ECC^T_*(Y,\lambda_{\varepsilon(T)},0,J)$ is generated by orbit sets $\{(\gamma_{p_1},m_1), (\gamma_{p_2},m_2)\}$ such that
$$(m_1+m_2)(1+\varepsilon \mathfrak{q}^*f) \leq T.$$
Moreover, we have $c_\tau(D) = 0$ for any disk $D$ bounding $\gamma_{p_1}$ or $\gamma_{p_2}$ since $\tau$ is a global trivialization of $\xi=\ker\lambda_0$ Let $Q_i := Q_\tau(Z_i)$, where $Z_i$ is the unique class in $H_2(Y, \gamma_{p_i},\emptyset)$ for $i=1,2$, and let $Q_{12}= Q_\tau(Z_1,Z_2)$. Then, from the definition of $Q_\tau$ and using equation \eqref{qtauquadratic} repeatedly, we obtain$$Q_\tau(\gamma_{p_1}^{m_1}\gamma_{p_2}^{m_2}) = m_1^2Q_1 + 2m_1m_2 Q_{12} + m_2^2 Q_2.$$
Furthermore, for this trivialization we have $Q_i = Q_i - c_\tau(Z_i)$ which agrees with the self-linking number of the transverse knot $\gamma_{p_i}$, and $Q_{12}$ is just the linking number $lk(\gamma_{p_1},\gamma_{p_2})$. Since $\gamma_{p_i}$ is a fiber of the Hopf fibration, we have $Q_i = -1$, and noting that $\gamma_1$ intersects once and positively a disk bounded by $\gamma_2$ (and vice-versa), we have $Q_{12} = 1$; see \cite{hryniewicz2014systems} for a more general situation that one has these numbers. Therefore
\begin{equation}\label{eq:qtau}Q_\tau(\gamma_{p_1}^{m_1}\gamma_{p_2}^{m_2}) = -m_1^2 +2m_1m_2 - m_2^2.\end{equation}
Combining \eqref{CZs3} and \eqref{eq:qtau}, we compute the ECH index between $\alpha = \{(\gamma_{p_1},m_1), (\gamma_{p_2},m_2)\}$ and the empty set:
$$I(\alpha, \emptyset) = -m_1^2 + 2m_1m_2 -m_2^2 + \sum_{k=1}^{m_1} 4k -1 + \sum_{k=1}^{m_2} 4k +1 = (m_1+m_2)^2 + m_1 + 3m_2.$$
This agrees with the ECH index in \cite[Proposition 3.5]{nelson2020embedded}. Taking the direct limit on $T$, we obtain the well-known ECH of $S^3$:
\begin{align*}
ECH_*(S^3,\xi_0,0) = \begin{cases}
\Z, & \text{if} \ * \in 2\Z_{\geq 0} \\
0, & \text{else},
\end{cases}
\end{align*}
and ECH spectrum
$$c_k(\partial B(1),\lambda_0) = c_k(B(1)) = N(1,1)_k.$$
The ECH spectrum comes from computing the actions of the generators ordered by the grading $|\alpha| = I(\alpha, \emptyset)$:
$$ 1, \gamma_{p_1}, \gamma_{p_2}, \gamma_{p_1}^2, \gamma_{p_1} \gamma_{p_2},  \gamma_{p_2}^2, \gamma_{p_1}^3, \gamma_{p_1}^2 \gamma_{p_2}, \gamma_{p_1} \gamma_{p_2}^2, \gamma_{p_2}^3, \ldots,$$
the $U$ map described in \cite[Proposition 4.1]{hutchings2014lecture}, and the fact that for our contact form we have
$$\mathcal{A}(\gamma_{p_1}^{m_1} \gamma_{p_2}^{m_2}) = m_1 + m_2.$$
We use the product notation $\gamma_{p_1}^{m_1} \gamma_{p_2}^{m_2}$ to denote an orbit set $\{(\gamma_{p_1},m_1), (\gamma_{p_2},m_2)\}$ above. In this case, we write $1$ for the empty set.

\section{Unit cotangent bundle of sphere}
\subsection{Preliminaries}
Let $\Sigma$ be a closed possibly non-orientable surface. It is well known that the cotangent bundle $\pi \colon T^*\Sigma \to \Sigma$ has a canonical symplectic form given by $\omega = d\lambda$, where $\lambda$ is the tautological $1$-form. Pick a Riemannian metric on $\Sigma$ and let $Y = S^*\Sigma$ be the unit cotangent bundle for this metric, i.e, the norm $1$ covectors of $\Sigma$. It turns out that $Y$ is a closed $3$-manifold, $\lambda$ restricts to a contact form on $Y$ (which we denote again by $\lambda$), and the Reeb flow on $(Y,\lambda)$ is dual to the geodesic flow on the unit tangent bundle
$$S\Sigma = \{v \in T\Sigma| \ \Vert v \Vert = 1\},$$
see e.g. \cite[Theorem 1.5.2]{geiges2008introduction}. In particular, the Reeb orbits correspond to closed orbits of geodesic flow on $S\Sigma$. Moreover, the length of a closed geodesic is equal to the action of the corresponding Reeb orbit. We denote the contact structure by $\xi = \ker \lambda$, and note that $\xi$ is always fillable since one can consider the disk cotangent bundle $(D^*\Sigma,\omega)$, i.e, the covectors on $\Sigma$ with norm $\leq 1$ equipped with the restriction of the canonical symplectic form.
If $(q_1,q_2)$ are local coordinates for $\Sigma$ there are induced cotangent coordinates $(q_1,q_2,p_1,p_2)$ such that $\lambda$ is locally given by
$$\lambda = p_1 dq_1 + p_2 dq_2$$
on $T^*\Sigma$. If $\Sigma$ is orientable, $\xi$ admits a non-vanishing global section with values in the fiber of the bundle $S^*\Sigma\to \Sigma$. Such a section can be constructed as follows. Suppose that $\{\partial_{q_1},\partial_{ q_2}\}$ is positively oriented. Let $A(q)$ be the matrix representing the Riemannian metric in $T^*\Sigma$. In particular, for $(q,p)\in T^*\Sigma$, we have  $\Vert p\Vert^2=\langle A(q)p,p\rangle$. So we define \[\partial_{\theta}=\frac{J_0 A(q) p}{\Vert J_0 A(q) p \Vert},\quad \text{where } J_0=\begin{bmatrix}0&-1\\1&0\end{bmatrix}. \] It follows from a straight-forward computation that a oriented change of coordinates preserves $\partial_\theta$ locally. So $\partial_\theta$ is globally defined. Thus $c_1(\xi) = 0$.
\begin{rmk}
Hutchings and Sullivan have given a thorough combinatorial description of ECH of $T^3 = S^1 \times T^2 = \R/2\pi\Z \times \R^2/\Z^2 = S^*T^2$ with the contact structure described above in \cite{hutchings2006rounding}.
\end{rmk}

We note here that the tautological form $\lambda$ on the unit cotangent bundle of a surface $Y = S^*\Sigma$ is Zoll if and only if the metric on $\Sigma$ which defines $Y$ is Zoll, i.e., all prime geodesics are closed and have the same length. In particular, there are exactly two cases in which $\lambda$ is a Zoll contact form, namely, $\Sigma = S^2$ and $\Sigma = \R P^2$. Our goal is to treat both cases in the rest of this work. First we recall that the orbit spaces are diffeomorphic to the sphere $S^2$.

\begin{lemma}\label{orbitspace2}
Let $Y = S^*S^2$ or $S^*\R P^2$ defined by a Zoll metric and $\lambda$ be the tautological contact form. Then the Reeb orbit space $\SS = Y/S^1$ is diffeomorphic to the sphere $S^2$.

\begin{proof}
By Lemma \ref{orbitspace}, $\SS$ is an orientable compact surface and so we prove this Lemma by checking its topology. The homotopy long exact sequence related to the fibration $S^1 \hookrightarrow Y \twoheadrightarrow \SS$ yields
$$\cdots \to \pi_2(S^1) \to \pi_2(Y) \to \pi_2(\SS) \to \pi_1(S^1) \to \pi_1(Y) \to \pi_1(\SS) \to \cdots,$$
in particular, $\pi_2(\SS) \cong \Z$, since $\pi_2(Y) = 0$ and $\pi_1(Y) = \Z_2$ or $\Z_4$ (see section \S \ref{total} and \S \ref{total2}). Therefore, $\SS$ is diffeomorphic to the sphere $S^2$.
\end{proof}
\end{lemma}

\subsection{The topology of $S^* S^2$} \label{total}
Let $Y = S^*S^2$. Using the inclusion $T^* S^2\subset T^*\R^3$, we have a natural identification
\begin{equation} \label{r6}
Y = \{(q,p) \in \R^3 \times \R^3\mid \Vert q\Vert=\Vert p\Vert = 1, \langle q,p \rangle = 0\},
\end{equation}
where $\Vert \cdot \Vert$ and $\langle \cdot, \cdot \rangle$ denote the Euclidean norm and the scalar product on $\R^3$, respectively.
Let $S^3\subset\mathbb{H}=\C^2$ be the unit quaternionic sphere. We define
\begin{eqnarray*}
\psi \colon S^3 \subset \mathbb{H} &\to& Y \subset \R^3 \times \R^3 \\
u &\mapsto& (uj\overline{u}, uk\overline{u}).
\end{eqnarray*}
Here $\overline{u}$ denotes the quaternionic conjugate of $u$.
\begin{lemma} \label{doblecov}
The map
\begin{eqnarray*}
\psi \colon S^3 &\to& Y \subset \R^3 \times \R^3 \\
u &\mapsto& (uj\overline{u}, uk\overline{u})
\end{eqnarray*} is a double covering such that $\psi^*\lambda=4\lambda_0$. In particular, $\psi$ lifts doubly covered Reeb orbits of $(Y,\lambda)$ to Reeb orbits of $(S^3,\lambda_0)$.
\end{lemma}
\begin{proof}
It is straight-forward to see that $\psi$ is a smooth double covering.
Note that the tautological contact form visualizing $Y \subset \R^3 \times \R^3$ as in \eqref{r6} is given by
$$\lambda_{(q,p)}(w_1,w_2) = \langle p,w_1 \rangle.$$
This way, inspired in \cite{albers2018reeb}, we compute
\begin{eqnarray*}
(\psi^*\lambda)_u(w_u) &=& \lambda_{\psi(u)}(d\psi_u(w_u)) \\ &=& \lambda_{(uj\overline{u},uk\overline{u})}(uj\overline{w_u}+w_uj\overline{u},uk\overline{w_u}+w_uku) \\ &=& \langle uk\overline{u}, uj\overline{w_u}+w_uj\overline{u} \rangle \\ &=& \langle k, j\overline{w_u}u+\overline{u}w_uj \rangle \\ &=& \text{Re}(i\overline{w_u}u-k\overline{u}w_uj) \\ &=& \text{Re}(i\overline{w_u}u+k\overline{w_u}uj) \\ &=& \text{Re}(i\overline{w_u}u+i\overline{w_u}u) \\ &=& 2\text{Re}(i\overline{w_u}u),
\end{eqnarray*}
for $u \in S^3 \subset \mathbb{H}$ and $w_u$ a tangent vector in $T_uS^3 \subset \mathbb{H}$. One then can check directly that
$$\psi^*\lambda = 2(x_1dy_1 - y_1dx_1 + x_2dy_2 - y_2dx_2)=4\lambda_0.$$
\end{proof}

It turns out that $\psi$ induces a well known diffeomorphism between $Y = S^*S^2$ and the real projective space $\R P^3$. Thus, we conclude
\begin{align}\label{homologys2}
H_*(Y;\Z) = \begin{cases}
\Z, & * = 0, 3 \\
\Z_2, & * = 1 \\
0, & * = \text{otherwise}.
\end{cases}
\end{align}
We note here that a smooth curve with transverse self-intersections in $Y$ is nullhomologus exactly when it has an odd number of self-intersections.

\subsection{Reeb orbits and perturbation}
A well known fact is that the closed geodesics of the standard sphere $S^2$ are exactly the great circles (and its iterations). So every geodesic is closed and has prime period equal to $2\pi$. We fix this standard metric as a defining metric for $Y=S^*S^2$. Therefore the tautological contact form $\lambda = pdq$ is Zoll and by Lemma \ref{orbitspace2}, the orbit space is homeomorphic to the sphere $S^2$. We note that, when we fix the action range, the perturbed contact form does not realize some Reeb orbits from the old one, but cannot create new orbits. In particular, the remaining Reeb orbits keep corresponding to closed geodesics on $S^2$.
\

In this case, we trivialize $\xi = \ker \lambda$ along an orbit $\gamma$ using $\tau$ given by $\partial_\theta$, obtaining the symplectic matrices path $\Psi_\gamma(t)$ given by
$$\Psi_\gamma(t) = \begin{pmatrix}
\cos t & -\sin t \\
\sin t & \cos t
\end{pmatrix}, \quad t \in [0,2\pi].$$
So the crossing form has only $2$ regular crossings, namely at the ending points $\Psi_\gamma(0) = \Psi_\gamma(2\pi) = I_{2 \times 2}$, and thus the Robin--Salamon index of a simple orbit is $2$, yielding $\mu^\tau(\gamma_p^k) =2k$. We note here that with our trivialization of $\xi$, the Robin--Salamon index of an orbit must agree with the Morse index of this orbit as a critical point of energy functional being a geodesic on the sphere, see \cite[Proposition 1.7.3]{eliashberg2000introduction}. Combining this with Lemma \ref{bourg}, for $T > 2\pi$ and small $\varepsilon = \varepsilon(T)>0$, the periodic orbits of $R_{\lambda_\varepsilon}$ on $Y$ with action up to $T$ are nondegenerate and lie above the critical points $p_1,p_2$ of $f$. So it follows from \eqref{CZ} that
\begin{equation} \label{CZesfera}
CZ_\tau(\gamma_{p_1}^k) = 2k -1 \quad \text{and} \quad CZ_\tau(\gamma_{p_2}^k) =2k + 1.
\end{equation}
So the Reeb orbits that generates the filtered chain complex $ECC^T_*(Y,\lambda_{\varepsilon(T)},\Gamma,J)$ are all elliptic again. As in \S \ref{esferas3}, the index parity property gives that the differential vanishes. Therefore the filtered ECH homology $ECH^T_*(Y,\lambda_{\varepsilon(T)},\Gamma,J)$ is isomorphic to its chain complex $ECC^T_*(Y,\lambda_{\varepsilon(T)},\Gamma,J)$, which is generated by the orbit sets $\{(\gamma_{p_1},m_1), (\gamma_{p_2},m_2)\}$ in homology class $\Gamma$ such that
$$2\pi(m_1+m_2)(1+\varepsilon \mathfrak{q}^*f) \leq T.$$

\subsection{ECH index}
The goal of this section is to compute the ECH index between two orbit sets and then establishing an absolute grading on $ECH^T_*(S^*S^2,\lambda_\varepsilon,\Gamma,J)$.
\

Given an orbit set $\alpha = \{(\gamma_{p_1},m_1), (\gamma_{p_2},m_2)\}$, its homology class is described as
\begin{equation} \label{ahomologia}
\alpha \in \Gamma \in H_1(Y) \cong \Z_2 \Leftrightarrow m_1+m_2 \equiv \Gamma\mod 2.
\end{equation}
Indeed, simple closed geodesics of sphere are never contractible in $Y = S^*S^2$ (see section \S \ref{total}). Now we compute the ECH index.

\begin{prop}\label{echindex2}
Let $\alpha = \{(\gamma_{p_1},m_1),(\gamma_{p_2},m_2)\}$ and~$\beta= \{(\gamma_{p_1},n_1),(\gamma_{p_2},n_2)\}$ be orbit sets in the same homology class $\Gamma$. Then the ECH index $I(\alpha,\beta)$ is given by
$$I(\alpha,\beta) = \frac{1}{2}\left(m_1+m_2-n_1-n_2\right)^2  + (m_1+m_2-n_1-n_2)(n_1+n_2) +2m_2 -2n_2.$$
\begin{proof}
First we note that since $H_2(Y) = 0$, there is a unique class $Z \in H_2(Y,\alpha,\beta)$. If $S$ is a $2$-chain representing this class, then
$$\partial S = m_1\gamma_{p_1} + m_2\gamma_{p_2} - n_1\gamma_{p_1} - n_2\gamma_{p_2} = (m_1-n_1)\gamma_{p_1} + (m_2-n_2)\gamma_{p_2}.$$
Therefore, if $S_0$ is a representative of the class $Z_0 \in H_2(Y,\gamma_{p_1}^{2(m_1-n_1)}\gamma_{p_2}^{2(m_2-n_2)},\emptyset)$, we can take $S_0 +n_1(\R \times \gamma_{p_1}^2) + n_2(\R \times \gamma_{p_2}^2)$ as a representative of $2Z$. Let $C_\gamma$ denote the trivial cylinder $\R \times \gamma$. Now using \eqref{qtauquadratic}, we compute
\begin{eqnarray}
Q_\tau(Z) &=& \frac{1}{4}Q_\tau(2Z) = \frac{1}{4} Q_\tau(S_0 + n_1 C_{\gamma_{p_1}^2} + n_2C_{\gamma_{p_2}^2}) \nonumber \\ &=&\frac{1}{4}\left(Q_\tau(S_0) + 2Q_\tau(S_0,n_1C_{\gamma_{p_1}^2}) + 2Q_\tau(S_0,n_2C_{\gamma_{p_2}^2})\right). \label{qtauZ}
\end{eqnarray}
Let $Z_i$ be the unique class in $H_2(Y,\gamma_{p_i}^2,\emptyset)$ for $i=1,2$. It follows from \eqref{qtauquadratic} that
\begin{eqnarray*}
Q_\tau(S_0) &=& (m_1-n_1)^2Q_\tau(Z_1) + 2(m_1-n_1)(m_2-n_2)Q_\tau(Z_1,Z_2) \\ &+&(m_2-n_2)^2Q_\tau(Z_2).
\end{eqnarray*}
Since $c_\tau$ vanishes in our trivialization, we have $Q_\tau(Z_i) = Q_\tau(Z_i) - c_\tau(Z_i)$, and thus, it coincides with the self-linking number $sl(\gamma_{p_i}^2)$. Moreover, $Q_\tau(Z_1,Z_2)$ is just the linking number $lk(\gamma_{p_1}^2,\gamma_{p_2}^2)$. To compute these linking numbers, we use the commutative diagram that we obtained in section \S \ref{total}.
\[
  \begin{tikzcd}
     S^3,4\lambda_0 \arrow[dd,bend right] \arrow{d} \arrow[swap]{dr} \\ \R P^3 \arrow{r}{\simeq} \arrow[d, dotted] & Y \arrow[swap]{dl}{\pi}   \\ S^2 
  \end{tikzcd}
\]
Note that the standard contact form $\lambda_0$ defined on $S^3$ is invariant by the antipodal map and then, defines a contact form on $\R P^3 \cong Y$. The Reeb orbits from this contact form are exactly the fibers of bundle map $\pi \colon Y = S^*S^2 \to S^2$ while the Reeb orbits from the tautological form $\lambda$ are transverse to these fibers. Lemma \ref{doblecov} ensures that $\gamma_{p_i}^2$ lifts to the image of a Hopf fiber in $S^3$. Hence, without loss, we can suppose that $\gamma_{p_1}^2$ and $\gamma_{p_2}^2$ lift to $\widetilde{\gamma}_1$, $\widetilde{\gamma}_2$ with images $\{(a,b,0,0)\}, \{(0,0,c,d)\}$ in $S^3 \subset \R^4$, respectively. Thus $sl(\widetilde{\gamma}_1) = sl(\widetilde{\gamma}_2) = -1$. A simple computation shows that $lk(\widetilde{\gamma}_1,\widetilde{\gamma}_2) = 1$, and so, we conclude
\begin{equation}\label{lks}
sl(\gamma_{p_i}^2) = -2, \ \text{for} \ i=1,2 \quad \text{and} \quad lk(\gamma_{p_1}^2,\gamma_{p_2}^2) = 2.
\end{equation}
Therefore,
$$Q_\tau(S_0) =  -2(m_1-n_1)^2 + 4(m_1-n_1)(m_2-n_2) -2(m_2-n_2).$$
Furthermore, the remaining terms on $Q_\tau(Z)$ are given by
\begin{eqnarray*}
Q_\tau(S_0,n_1C_{\gamma_{p_1}^2}) &=& (m_1-n_1)n_1 sl(\gamma_{p_1}^2) + (m_2-n_2)n_1 lk(\gamma_{p_1}^2,\gamma_{p_2}^2) \\ &=& -2(m_1-n_1)n_1 + 2(m_2-n_2)n_1
\end{eqnarray*}
and
\begin{eqnarray*}
Q_\tau(S_0,n_2C_{\gamma_{p_2}^2}) &=& (m_2-n_2)n_2 sl(\gamma_{p_2}^2) + (m_1-n_1)n_2 lk(\gamma_{p_1}^2,\gamma_{p_2}^2) \\ &=& -2(m_2-n_2)n_2 + 2(m_1-n_1)n_2.
\end{eqnarray*}
Finally, putting these values in  \eqref{qtauZ}, yields
\begin{eqnarray*}
Q_\tau(Z) &=& \frac{1}{4}(-2(m_1-n_1)^2 + 4(m_1-n_1)(m_2-n_2) -2(m_2-n_2) \\ &+& 2(-2(m_1-n_1)n_1 + 2(m_2-n_2)n_1)  \\ &+& 2(-2(m_2-n_2)n_2 + 2(m_1-n_1)n_2)) \\ &=& -\frac{m_1^2}{2} - \frac{m_2^2}{2} + \frac{n_1^2}{2} + \frac{n_2^2}{2} +m_1m_2 - n_1n_2.
\end{eqnarray*}
Recalling the Conley-Zehnder indices in \eqref{CZ}, we obtain
\begin{eqnarray}
I(\alpha,\beta) &=& c_\tau(Z) +Q_\tau(Z) +CZ_\tau(\alpha) - CZ_\tau(\beta) \nonumber \\ &=& -\frac{m_1^2}{2} - \frac{m_2^2}{2} + \frac{n_1^2}{2} + \frac{n_2^2}{2} +m_1m_2 - n_1n_2 \nonumber \\ &+& \sum_{j=1}^{m_1} 2j-1 + \sum_{j=1}^{m_2} 2j+1 - \sum_{j=1}^{n_1} 2j-1 - \sum_{j=1}^{n_2} 2j+1 \nonumber \\ &=& \frac{m_1^2}{2} +\frac{m_2^2}{2} -\frac{n_1^2}{2} - \frac{n_2^2}{2} +2m_2 - 2n_2 +m_1m_2 - n_1n_2, \label{done}
\end{eqnarray}
since $\sum_{j=1}^a 2j-1 = a^2$ and $\sum_{j=1}^a 2j+1 = a^2+2a$. The expression in \eqref{done} is equal to the expression in the statement.
\end{proof}
\end{prop}

As we noted in \S \ref{grading}, we can define an absolute $\Z$ grading on $ECC_*^T(Y,\lambda_{\varepsilon(T)},0)$ by letting
$$|\alpha| = I(\alpha,\emptyset).$$
Thus, if $\alpha=\{(\gamma_{p_1},m_1),(\gamma_{p_2},m_2)\}$, it follows from \eqref{ahomologia} and Proposition \ref{echindex2}, that
\begin{equation} \label{grading2}
|\alpha| = \frac{1}{2}(m_1+m_2)^2 + 2m_2,
\end{equation}
whenever $m_1+m_2$ is an even number. Now it is a direct verification from equation \eqref{grading2} that the sequence of generators in homology class $\Gamma = 0$ ordered by (increasing) grading and starting at index $0$ is
\begin{equation}\label{ordergen}
1, \gamma_{p_1}^2,\gamma_{p_1} \gamma_{p_2}, \gamma_{p_2}^2, \gamma_{p_1}^4, \gamma_{p_1}^3 \gamma_{p_2}, \gamma_{p_1}^2 \gamma_{p_2}^2, \gamma_{p_1} \gamma_{p_2}^3, \gamma_{p_2}^4, \ldots
\end{equation}

\subsection{Direct limit}
In this section, we follow \cite{nelson2020embedded} and summarize the computation of $ECH_*(S^*S^2,\xi,\Gamma)$. All the results here follows directly as the analogous results in the reference.

\begin{prop}(cf. \cite[Proposition 3.2] {nelson2020embedded}) \label{dirlim}
Let $Y, \lambda$ and $\varepsilon(T)$ as discussed previously, for any $\Gamma \in \Z_2 \cong H_1(Y)$, the filtered embedded contact homology groups $ECH_*^T(Y, \lambda_{\varepsilon(T)}, \Gamma)$ form a direct system. The direct limit $\lim_{T\to\infty}ECH_*^T(Y, \lambda_\varepsilon, \Gamma)$ is the homology of the chain complex generated by the orbit sets $\{(\gamma_{p_1},m_1),(\gamma_{p_2},m_2)\}$, such that $m_1+m_2 \equiv \Gamma \mod 2$.
\end{prop}

\begin{prop}(cf. \cite[Theorem 7.1]{nelson2020embedded}) \label{limit}
For $Y = S^*S^2$, the contact form $\lambda_\varepsilon$, and $\varepsilon = \varepsilon(T)$ as in \eqref{lambdaepsilon}, and Lemma \ref{bourg}, respectively,
$$\lim_{T\to\infty}ECH_*^T(Y, \lambda_\varepsilon, \Gamma) = ECH_*(Y,\lambda, \Gamma),$$
for both $\Gamma \in H_1(Y) \cong \Z_2$.
\end{prop}

\begin{prop}\label{echsphere}
Let $S^*S^2$ be the unit cotangent bundle of the round sphere with the standard metric, $\lambda$ be the tautological form, and $\xi = \ker \lambda$. Then the $\Z$-graded $ECH$ of it is given by
\begin{align*}
ECH_*(S^*S^2,\xi,\Gamma) = \begin{cases}
\Z, & \text{if} \ * \in 2\Z_{\geq 0} \\
0, & \text{else}
\end{cases}
\end{align*}
for each $\Gamma \in H_1(S^*S^2)$.
\end{prop}

\begin{proof}
By Proposition \ref{dirlim} and Proposition \ref{limit} and recalling that our filtered chain complex coincides with the filtered ECH homology, it is enough to show that the ECH index in Proposition \ref{echindex2} gives a bijection between the chain complex generators and the nonnegative even integers $2\Z_{\geq 0}$. Our index agrees with the ECH index in \cite[Proposition 3.5]{nelson2020embedded} for $Y = L(2,1)$ as a prequantization bundle, so this bijection follows from the one found in \cite[Theorem 7.6]{nelson2020embedded} with Euler number $e=-2$.
\end{proof}

\begin{rmk}\label{rmkz}
Since $S^*S^2$ and $L(2,1)$ are diffeomorphic, $ECH_*(S^*S^2,\xi,\Gamma)$ can be computed using the isomorphism with Seiberg--Witten Floer cohomology and the results in \cite{kronheimer2007monopoles}. Alternatively, it follows from \cite{abbondandolo2017systolic} that the orbit space quotient $\mathfrak{q}\colon S^*S^2 \to \SS \cong S^2$ is a prequantization bundle. So Proposition \ref{echsphere} follows from \cite{nelson2020embedded}.
\end{rmk}

\subsection{U map and ECH spectrum}\label{umapspec}
The goal of this section is to describe the $U$ map for the direct limit of the direct system consisting in our filtered ECH groups $ECH_*^T(Y, \lambda_{\varepsilon(T)},0)$. After that, we compute the ECH spectrum $c_k(S^*S^2,\lambda)$ by means of a limiting argument.
\

Note that Proposition \ref{limit} and Proposition \ref{echsphere} show that $\lim_{T\to\infty}ECH_*^T(Y, \lambda_\varepsilon,0)$ has exactly one generator of grading $2k$ for each $k \in \Z_{\geq 0}$; denote this generator by $\zeta_k$. Moreover, as we mentioned in Remark \ref{rmk:umap}, our $U$ map coincides with the $U$ map on Seiberg-Witten Floer cohomology of the lens space $L(2,1)$ whose is already known \cite{kronheimer2007monopoles,kronheimer2007monopolesb}. Hence, we have the following Proposition. 

\begin{prop}\label{Umap1}
For any almost complex structure $J$, the $U$ map on the direct limit is given by
$$U\zeta_k = \zeta_{k-1}, \quad k \in \Z_{\geq 1}.$$
\end{prop}
Now we present a second approach to prove Proposition \ref{Umap1} following closely the arguments in \cite[\S 4.1]{hutchings2014lecture}. In fact, this follows from the generators grading order in \eqref{ordergen} and the next Lemma.

\begin{lemma} \label{Umap2}
The $U$ map on $\lim_{T\to\infty}ECH_*^T(Y, \lambda_\varepsilon, 0)$ is given by:
\begin{enumerate}[label=\alph*)]
\item $U(\gamma_{p_1}^i \gamma_{p_2}^j) = \gamma_{p_1}^{i+1}\gamma_{p_2}^{j-1}$, if $j >0$. \label{(a)}
\item $U(\gamma_{p_1}^i) = \gamma_{p_2}^{i-2}$, if $i>0$. \label{(b)}
\end{enumerate}
\begin{proof}
The independence of $J$ follows from the same argument in \cite[Proposition 4.1]{hutchings2014lecture}. Let $i+j > 0$. If $\mathcal{C}$ is a $J$-holomorphic current counted in $U(\gamma_{p_1}^i\gamma_{p_2}^j)$ then, by Proposition \ref{lowind}, we can write $\mathcal{C} = C_0 \sqcup C_2$, where $C_0$ is a union of trivial cylinders with multiplicities, and $C_2$ is embedded with $\text{ind}(C_2)=I(C_2) = 2$. We ignore the trivial cylinders part $C_0$. Moreover, we can take a representative in our direct system such that the rotation angles of $\gamma_{p_1}$ and $\gamma_{p_2}$, with respect the trivialization $\tau$ of $\xi|_{\gamma_{p_1} \cup \gamma_{p_2}}$, represent classes in $\R/\Z$ very close to $0$ and $1$, respectively. Partition conditions (see \cite[section \S 3.9]{hutchings2014lecture}) guarantee that $C_2$ has at most one positive end at a cover of $\gamma_{p_1}$, all negative ends of $C_2$ at covers of $\gamma_{p_1}$ have multiplicity $1$, and the opposite for $\gamma_{p_2}$. \\ 
\textbf{Case $j=0$:} Let $i>2$. By the latter fact and generators grading order \eqref{ordergen}, $C_2$ has exactly one negative end at $\gamma_{p_2}^{i-2}$, yielding
\begin{eqnarray*}
2 = \text{ind}(C_2) &=& -\chi(C_2) + CZ_\tau(\gamma_{p_1}^i) -  CZ_\tau(\gamma_{p_2}^{i-2}) \\ &=&-\chi(C_2) + 2i - 1 - (2(i-2)+1) \\ &=& -\chi(C_2) + 2,
\end{eqnarray*}
and thus $C_2$ is a cylinder between with positive end at $\gamma_{p_1}^i$ and negative end $\gamma_{p_2}^{i-2}$. If $i=2$, Fredholm index equation yields $\chi(C_2) = 1$ and so, $C_2$ is a plane bounding $\gamma_{p_1}^2$. \\
\textbf{Case $j>0$:} First let $i=0$ and $j>1$. Then $C_2$ has $j$ positive ends at $\gamma_{p_2}$, a negative end at $\gamma_{p_2}^{j-1}$ and a negative end at $\gamma_{p_1}$. Then
\begin{eqnarray*}
2 = \text{ind}(C_2) &=& -\chi(C_2) + 3j - (2(j-1)+1 + 1) \\ &=& 2g(C_2) - 2 +j +2 + j,
\end{eqnarray*}
which leads to the contradiction $g(C_2) = 1-j<0$. \\
If $i,j>1$, $C_2$ has a positive end at $\gamma_{p_1}^i$, $j$ positive ends at $\gamma_{p_2}$, $i+1$ negative ends at $\gamma_{p_1}$ and a negative end at $\gamma_{p_2}^{j-1}$. So
\begin{eqnarray*}
2 = \text{ind}(C_2) &=& -\chi(C_2) + 3j + 2i-1 - (2(j-1)+1 + i+1) ) \\ &=& 2g(C_2) - 2 +j +i+3 + j + i - 1 \\ &=& 2g(C_2) +2i+2j,
\end{eqnarray*}
yielding the contradiction $g(C_2) = 1 - i - j<0$. \\
If $j=1$ and $i>0$, $C_2$ has a positive end at $\gamma_{p_1}^i$, a positive end at $\gamma_{p_2}$ and $i+1$ negative ends at $\gamma_{p_1}$. This way,
\begin{eqnarray*}
2 = \text{ind}(C_2) &=& -\chi(C_2) + 3 + 2i-1 - (i+1) \\ &=& 2g(C_2) - 2 +i+3 + 1 + i \\ &=& 2g(C_2) +2i + 2,
\end{eqnarray*}
giving the contradiction $g(C_2) = 1-i<0$. \\
The remaining possibility is $i=0$ and $j=1$. In this case, $C_2$ has a positive end at $\gamma_{p_2}$ and a negative end at $\gamma_{p_1}$, thus
$$2 = \text{ind}(C_2) = -\chi(C_2) + 3 -1 = -\chi(C_2) +2,$$
concluding that $C_2$ is a cylinder from $\gamma_{p_2}$ to $\gamma_{p_1}$. \\
Therefore, to prove this Lemma we need to count $J$-holomorphic cylinders $C$ with a positive end at $\gamma_{p_1}^i$, and a negative end at $\gamma_{p_2}^{i-2}$, planes that bound $\gamma_{p_1}^2$, and cylinders from $\gamma_{p_2}$ to $\gamma_{p_1}$, all of these passing through a fixed base point $(0,z) \in \R \times Y$ such that $z$ is not on any Reeb orbit. It follows from \cite{bourgeois2002morse} and \cite{moreno} that the latter counting is equivalent to counting negative gradient flow lines of $f\colon \SS = S^2 \to \R$ from $p_2$ to $p_1$ passing through a fixed point in $S^2$. Since $z \in Y$ is not on a Reeb orbit, the corresponding point in $S^2$ cannot be $p_1$ or $p_2$. Thus there is a unique flow line that counts and hence, a unique cylinder with a positive end at $\gamma_{p_2}$ and a negative end at $\gamma_{p_1}$ passing through $(0,z)$, so proving the part \ref{(a)}. \

For the remaining counting, we note that the symplectization $\R \times S^*S^2$ can be identified with $T^*S^2\backslash S^2$ viewing $S^2 \subset T^*S^2$ as the zero section, so that the canonical complex structure on $T^*S^2$ corresponds to a symplectization-admissible almost complex structure, and under this identification, our count is equivalent to count the number of meromorphic sections of $T^*S^2$, i.e., meromorphic $1$-forms on $S^2$, having a pole of order $i$ at $p_1$, a zero of order $i-2$ at $p_2$, $i\geq 2$, and no other zeros or poles, passing through a base point in  $T^*S^2\backslash S^2$. By Riemann-Roch theorem, the last number is one, and we are done.
\end{proof}
\end{lemma}

Proposition \ref{Umap1} give us that the $U$ map sends the generator in grading $2k$, $\zeta_k$, to the generator in grading $2k-2$, $\zeta_{k-1}$, for $k \in \Z_{\geq 1}$. So we can compute the ECH spectrum using the generators grading order
\begin{equation*}
1, \gamma_{p_1}^2,\gamma_{p_1} \gamma_{p_2}, \gamma_{p_2}^2, \gamma_{p_1}^4, \gamma_{p_1}^3 \gamma_{p_2}, \gamma_{p_1}^2 \gamma_{p_2}^2, \gamma_{p_1} \gamma_{p_2}^3, \gamma_{p_2}^4, \ldots,
\end{equation*}
and the fact that for our original contact form $\lambda$ we have
$$\mathcal{A}(\gamma_{p_1}^{m_1}\gamma_{p_2}^{m_2}) =2\pi(m_1 + m_2).$$
In fact, we recall \eqref{cklim}, and do
\begin{eqnarray*}
c_k(Y,\lambda) &=& \lim_{\varepsilon\to 0} c_k(Y,(1+\varepsilon\mathfrak{q}^*f)\lambda) \\ &=& 2\pi M_2(N(1,1))_k,
\end{eqnarray*}
which proves part \ref{parta} of Theorem \ref{thm:capacities}.

\section{Unit tangent bundle of real projective plane}
Now we explain in a short way how can one does prove essentially the same things that we proved in sphere case but in case of real projective plane $\R P^2$. First of all, recall that $\R P^2$ inherits a standard Zoll metric induced by the standard metric on $S^2$ and the quotient map
\begin{eqnarray*}
q\colon S^2 &\to& \R P^2 = S^2/(p\sim -p) \\
p &\mapsto& [p].
\end{eqnarray*}
We fix this metric to define $Y=S^* \R P^2$. In this case, the tautological contact form $\lambda$ is Zoll and the Reeb orbits have minimal action equal to $\pi$. From now on, we denote the tautological contact form on $S^*S^2$ by $\tilde{\lambda}$ and the associated contact structure by $\tilde{\xi} = \ker \tilde{\lambda}$.

\subsection{Topology} \label{total2}
Now to describe the topology of the total space of the $S^1$-bundle $\pi \colon Y \to \R P^2$, we note that the $2$-covering map $q$ induces a $2$-covering map
\begin{eqnarray*}
\tilde{q}\colon S^*S^2 &\to& Y = S^*\R P^2
\end{eqnarray*}
such that $\tilde{q}^*\lambda = \tilde{\lambda}$. Furthermore, since $S^*S^2 \cong \R P^3 = L(2,1)$, we have a commutative diagram
\[\begin{tikzcd}
S^3 \arrow{d}{2:1} \arrow{dr}{4:1} \\ L(2,1) \arrow{r}{2:1} \arrow[d] & Y  \arrow[d] \\
\ S^2 \arrow{r}{2:1}  & \R P^2
\end{tikzcd}
\]
and in fact, it turns out that $Y$ is diffeomorphic to the lens space $L(4,1)$, see \cite{konno2002unit}. In particular,
\begin{align}\label{homology}
H_*(Y;\Z) = \begin{cases}
\Z, & * = 0, 3 \\
\Z_4, & * = 1 \\
0, & * = \text{otherwise}.
\end{cases}
\end{align}
Moreover, we recall that a simple closed geodesic on $\R P^2$ lifts to half of a great circle on $S^2$, while a doubly covered simple geodesic lifts to a simple geodesic (a whole great circle) on the sphere via $q$. Therefore, a doubly covered Reeb orbit $\gamma^2$ on $Y$ lifts to a simple Reeb orbit $\tilde{\gamma}$ on $(S^*S^2,\tilde{\lambda})$ via $\tilde{q}$. Moreover, since $\tilde{\gamma}$ generates $H_1(S^*S^2) \cong \Z_2$, we have $2[\gamma]$ nonzero in $H_1(Y) \cong \Z_4$, and hence, $[\gamma]$ is a generator of $H_1(Y)$ for any simple Reeb orbit $\gamma$.

\subsection{ECH Index}
We consider again the perturbed contact form $\lambda_\varepsilon = (1+\varepsilon \mathfrak{q}^*f)$, with the quotient map $\mathfrak{q} \colon Y \to \SS$. Lemma \ref{orbitspace2} ensures that $\SS \cong S^2$ again, and thus, the filtered chain complex $ECC_*^T(Y,\lambda_{\varepsilon},J,\Gamma)$ is generated by two orbits $\gamma_{p_1},\gamma_{p_2}$. Furthermore, by last section, the homology class of an orbit set $\alpha = \{(\gamma_{p_1},m_1),(\gamma_{p_2},m_2)\}$ is determined by
\begin{equation}\label{ahomologia2}
\alpha \in \Gamma \in H_1(Y) \cong \Z_4 \Leftrightarrow m_1+m_2 \equiv \Gamma\mod 4.
\end{equation} 

The main difference now is that since $\R P^2$ is nonorientable, the previous global nonvanishing section $\partial_\theta \colon Y \to \xi = \ker \lambda$ does not exist any more. The problem here is that it is not true that every trivialization along a simple Reeb orbit on $S^*S^2$ descends to a trivialization along the image Reeb orbit on $Y$ via $\tilde{q}$. However, we can do a trick. Let $\tau$ be a trivialization along a Reeb orbit $\gamma$ on $Y$ such that the linearization of the Reeb flow turns around once along $\gamma$. Via the $2$-covering $\tilde{q}$, one can lift $\tau$ to a trivialization $\tilde{\tau}$ along a lift $\tilde{\gamma}$ of $\gamma$ on $S^*S^2$ such that the linearization of the Reeb flow now turns twice along $\tilde{\gamma}$.\

We do this previous idea along the orbits $\gamma_{p_1}$ and $\gamma_{p_2}$ and, without loss, we suppose that $\gamma_{p_1}^2$ and $\gamma_{p_2}^2$ lift to the Reeb orbits that we used in sphere case, now denoted by $\tilde{\gamma}_{p_1}$ and $\tilde{\gamma}_{p_2}$, respectively. In this case, the difference between the previous trivialization, now denoted by $\tau_\theta$, and this new trivialization $\tilde{\tau}$ is $\tilde{\tau} - \tau_\theta = 1$. Therefore, we can use equations \eqref{czmuda}, \eqref{ctaumuda} and \eqref{qtaumuda} to compute the new terms that appear in ECH index with this new trivialization. First, \eqref{czmuda} and \eqref{CZesfera} yield
\begin{equation}\label{CZesfera2}
CZ_{\tilde{\tau}}(\tilde{\gamma}_{p_1}^k) = 4k-1 \quad \text{and} \quad CZ_{\tilde{\tau}}(\tilde{\gamma}_{p_2}^k) = 4k+1.
\end{equation}
Now let $\tilde{Z}_i$ be the unique class in $H_2(S^*S^2,\tilde{\gamma}_{p_i}^2,\emptyset)$. Recalling that $c_{\tau_\theta}(\tilde{Z}_i) = 0$, equation \eqref{ctaumuda} gives us
\begin{equation}\label{ctauesfera2}
c_{\tilde{\tau}}(\tilde{Z}_i) = -2,
\end{equation}
and since $Q_{\tau_\theta}(\tilde{Z}_i) = sl(\tilde{\gamma_i}^2) = -2$, equations \eqref{qtauquadratic} and \eqref{qtaumuda} yield
\begin{equation}\label{qtauesfera2}
Q_{\tilde{\tau}}(\tilde{Z}_i) = -6.
\end{equation}
Therefore, we can compute these terms in $\R P^2$ case noting that if $\gamma$ is a Reeb orbit of $(Y,\lambda)$ such that $\gamma^2$ lifts to a Reeb orbit $\tilde{\gamma}$ on $(S^*S^2,\lambda)$, then the covering $\tilde{q}$ induces an isomorphism between $\tilde{\xi}|_{\tilde{\gamma}}$ and $\xi|_{\gamma^2}$. With this in mind, we have
$$\mu^\tau(\gamma^2) = \mu^{\tilde{\tau}}(\tilde{\gamma}) = 4,$$
and thus, $\mu^\tau(\gamma^k) = 2k$. Thus,
\begin{equation}\label{czrp2}
CZ_\tau(\gamma_{p_1}^k) = 2k-1 \quad \text{and} \quad CZ_\tau(\gamma_{p_2}^k) = 2k+1,
\end{equation}
by equations in \eqref{CZ}. In particular, the orbits $\gamma_{p_1}$ and $\gamma_{p_2}$ are elliptic, and hence, the filtered ECH homology $ECH_*^T(Y,\lambda_{\varepsilon(T)},J,\Gamma)$ is, for any almost complex structure $J$, isomorphic to its chain complex $ECC_*^T(Y,\lambda_{\varepsilon(T)},J,\Gamma)$ which is generated by orbit sets $\{(\gamma_{p_1},m_1),(\gamma_{p_2},m_2)\}$ in homology class $\Gamma$ such that
$$\pi(m_1+m_2)(1+\varepsilon \mathfrak{q}^*f) \leq T.$$  Moreover, if $Z_i$ is the unique class in $H_2(Y,\gamma_{p_i}^4,\emptyset)$, equations \eqref{ctauesfera2} and \eqref{qtauesfera2} yield
\begin{equation}\label{ctaurp2}
c_\tau(Z_i) = c_{\tilde{\tau}}(\tilde{Z}_i) = -2,
\end{equation}
and
\begin{equation}\label{qtaurp2}
Q_\tau(Z_i) = 2 Q_{\tilde{\tau}}(\tilde{Z_i}) = -12.
\end{equation}
Finally, let $\alpha = \{(\gamma_{p_1},m_1),(\gamma_{p_2},m_2)\}$ and $\beta= \{(\gamma_{p_1},n_1),(\gamma_{p_2},n_2)\}$ be two orbit sets, and let $Z \in H_2(Y,\alpha,\beta)$. Doing essentially the same that we did in the proof of Proposition \ref{echindex2} but changing $\gamma_{p_i}^2$ to $\gamma_{p_i}^4$, linearity of $c_\tau$ and quadratic property of $Q_\tau$, i.e., equations \eqref{ctaulinear} and \eqref{qtauquadratic}, together \eqref{czrp2}, \eqref{ctaurp2} and \eqref{qtaurp2} lead us to
\begin{equation*}
CZ_\tau(\alpha) = \sum_{j=1}^{m_1} 2k-1 + \sum_{j=1}^{m_2} 2k+1 = m_1^2 + m_2^2 - 2m_2,
\end{equation*}

\begin{equation*}
CZ_\tau(\beta) = \sum_{j=1}^{n_1} 2k-1 + \sum_{j=1}^{n_2} 2k+1 = n_1^2 + n_2^2 - 2n_2,
\end{equation*}

\begin{equation*}
c_\tau(Z) = -\frac{1}{2}(m_1+m_2 -n_1-n_2)
\end{equation*}
and
\begin{equation*}
Q_\tau(Z) = -\frac{3}{4}m_1^2 -\frac{3}{4}m_2^2 + \frac{3}{4}n_1^2 + \frac{3}{4}n_2^2 +\frac{1}{2}m_1m_2  -\frac{1}{2}n_1n_2,
\end{equation*}
Putting last equations together, we compute the ECH index $I(\alpha,\beta)$.

\begin{prop}\label{ECHindexrp2}
For any $T$ and $\Gamma$, given two orbit sets $\alpha = \{(\gamma_{p_1},m_1),(\gamma_{p_2},m_2)\}$ and $\beta= \{(\gamma_{p_1},n_1),(\gamma_{p_2},n_2)\}$, the ECH index in $ECH_*^T(S^*\R P^2,\lambda_{\varepsilon(T)},\Gamma)$ is given by
$$I(\alpha,\beta) = \frac{1}{4}(m_1+m_2-n_1-n_2)^2 + \frac{1}{2}(m_1+m_2-n_1-n_2)(n_1+n_2) -\frac{m_1}{2} + \frac{3m_2}{2} + \frac{n_1}{2} - \frac{3n_2}{2}.$$
\end{prop}
Here we can almost repeat Remark \ref{rmkz} changing $L(2,1)$ to $L(4,1)$. Moreover, the ECH index of Proposition \ref{ECHindexrp2} is equal to the index in \cite{nelson2020embedded} in the case of the lens space $L(4,1)$ as a prequantization bundle over the sphere $S^2$. Moreover, versions of Propositions \ref{dirlim} and \ref{limit}, follow again from Nelson-Weiler work. Therefore, we have
$$\lim_{T\to\infty}ECH_*^T(Y, \lambda_\varepsilon, \Gamma) = ECH_*(Y,\lambda, \Gamma),$$
for $\varepsilon = \varepsilon(T)$ as is Lemma \ref{bourg} and for each $\Gamma \in H_1(Y) \cong \Z_4$. Furthermore, the direct limit $\lim_{T\to\infty}ECH_*^T(Y, \lambda_\varepsilon, \Gamma)$ is the homology of the chain complex generated by the orbit sets $\{(\gamma_{p_1},m_1),(\gamma_{p_2},m_2)\}$ such that $m_1+m_2 \equiv \Gamma \mod 4$. Since $\gamma_{p_1}$ and $\gamma_{p_2}$ are elliptic, the differential vanishes, and hence, the bijection found in \cite[Theorem 7.6]{nelson2020embedded} for Euler number $e=-4$ proves the following Proposition.

\begin{prop}\label{echrp2}
Let $S^*\R P^2$ be the unit cotangent bundle of the real projective plane with the standard metric, $\lambda$ be the tautological form, and $\xi = \ker \lambda$. Then the $\Z$-graded $ECH$ of it is given by
\begin{align*}
ECH_*(S^* \R P^2,\xi,\Gamma) = \begin{cases}
\Z, & \text{if} \ * \in 2\Z_{\geq 0} \\
0, & \text{else}
\end{cases}
\end{align*}
for each $\Gamma \in H_1(S^* \R P^2)$.
\end{prop}

\subsection{U map and ECH spectrum}
Now we do the analogue of \S \ref{umapspec} for $Y = S^*\R P^2$. We already know that our ECH is given as the limit $\lim_{T\to\infty}ECH_*^T(Y, \lambda_\varepsilon, \Gamma)$ and it has exactly one generator of grading $2k$ for each $k \in \Z_{\geq 0}$, and vanishes else. Let $\zeta_k$ denote the generator of grading $2k$ of $\lim_{T\to\infty}ECH_*^T(Y, \lambda_\varepsilon, 0)$. Note that it follows from Proposition \ref{ECHindexrp2} that the grading $|\alpha| = I(\alpha,\emptyset)$ is
$$|\alpha| = \frac{1}{4}(m_1+m_2)^2 -\frac{m_1}{2} + \frac{3m_2}{2},$$
and thus, this together \eqref{ahomologia2}, ensure that the sequence of generators in homology class $\Gamma = 0$ ordered by (increasing) grading $|\alpha| = I(\alpha,\emptyset)$, and starting at index $0$ is given by

\begin{equation}\label{ordergen2}
1, \gamma_{p_1}^4, \gamma_{p_1}^3\gamma_{p_2}, \gamma_{p_1}^2 \gamma_{p_2}^2, \gamma_{p_1}\gamma_{p_2}^3, \gamma_{p_2}^4, \gamma_{p_1}^8, \gamma_{p_1}^7\gamma_{p_2}, \gamma_{p_1}^6\gamma_{p_2}^2, \gamma_{p_1}^5\gamma_{p_2}^3,\ldots
\end{equation}

\begin{prop}\label{Umap3}
For any almost complex structure $J$, the $U$ map on the direct limit is given by
$$U\zeta_k = \zeta_{k-1}, \quad k \in \Z_{\geq 1}.$$
\begin{proof}
It follows again from the $U$ map on Seiberg-Witten Floer Cohomology of lens spaces found in \cite{kronheimer2007monopoles,kronheimer2007monopoles}. For the sake of completeness, we sketch another approach here. We assert that the following holds.
\begin{enumerate}[label=\alph*)]
\item $U(\gamma_{p_1}^i \gamma_{p_2}^j) = \gamma_{p_1}^{i+1}\gamma_{p_2}^{j-1}$, if $j >0$. \label{(a2)}
\item $U(\gamma_{p_1}^i) = \gamma_{p_2}^{i-4}$, if $i>0$. \label{(b2)}
\end{enumerate}
The proof of Proposition \ref{Umap3} follows then from assertions \ref{(a2)} and \ref{(b2)} and the generators grading order in \eqref{ordergen2}. The independence of $J$ follows from the same argument in \cite[Proposition 4.1]{hutchings2014lecture}. Let $i+j >0$ and let $\mathcal{C}$ be a $J$-holomorphic current counted in $U(\gamma_{p_1}^i\gamma_{p_2}^j)$. Once more, we can write $\mathcal{C} = C_0 \sqcup C_2$, where $C_2$ is embedded with $\text{ind}(C_2) = I(C_2) = 2$ and that is the only part that matters in the counting. Again, partition conditions guarantee that $C_2$ has at most one positive end at a cover of $\gamma_{p_1}$, all negative ends of $C_2$ at covers of $\gamma_{p_1}$ have multiplicity $1$, and the opposite for $\gamma_{p_2}$.
\

Using equation $\text{ind}(C_2) = 2$, by computations very close to those in the proof of Lemma \ref{Umap2}, it turns out that we need to count $J$-holomorphic cylinders $C$ with a positive end at $\gamma_{p_1}^i$ and a negative end at $\gamma_{p_2}^{i-4}$, planes that bound $\gamma_{p_1}^4$, and cylinders from $\gamma_{p_2}$ to $\gamma_{p_1}$, all of these passing through a fixed point $(0,z) \in \R \times Y$ such that $z$ is not on any Reeb orbit. Similarly to the proof of Lemma \ref{Umap2}, a $J$-holomorphic cylinder with positive end at $\gamma_{p_2}$ and a negative end at $\gamma_{p_1}$ passing through $(0,z)$ corresponds to a negative gradient flow line of $f \colon \SS = S^2 \to \R$ from $p_2$ to $p_1$ passing through a point in $S^2 \backslash \{p_1,p_2\}$, and vice-versa. In particular, there is a unique cylinder with this properties, and thus, we obtain part \ref{(a2)}. \

To conclude \ref{(b2)}, we note that an symplectization admissible almost complex structure $\tilde{J}$ in $\R \times S^*S^2 \cong T^*S^2\backslash S^2$ induces an symplectization admissible almost complex structure $J$ in $\R \times S^*\R P^2$ by means of the covering $\tilde{q}$. Furthermore, a $J$-holomorphic cylinder between\footnote{We recall that since $\gamma_{p_1}^i$ is nullhomologous, $i$ is a multiple of $4$.} $\gamma_{p_1}^i$ and $\gamma_{p_2}^{i-4}$ lifts to a $\tilde{J}$-holomorphic cylinder between $\tilde{\gamma}_{p_1}^{i/2}$ and $\tilde{\gamma}^{i/2 - 2}$, while a $J$-holomorphic plane bounding $\gamma_{p_1}^4$ lifts to a $\tilde{J}$-holomorphic plane bounding $\gamma_{p_1}^2$ in $\R \times S^*S^2$. Therefore, our countings follow from the sphere case.
\end{proof}
\end{prop}
Finally, as we did for $S^*S^2$, the ECH spectrum follows from Proposition \ref{Umap3}, the generators grading order in \eqref{ordergen2}, and the fact that the action of our unperturbed contact form is given by
$$\mathcal{A}(\gamma_{p_1}^{m_1}\gamma_{p_2}^{m_2}) = \pi(m_1+m_2).$$
Therefore
\begin{eqnarray*}
c_k(Y,\lambda) &=& \lim_{\varepsilon\to 0} c_k(Y,(1+\varepsilon\mathfrak{q}^*f)\lambda) \\ &=& \pi M_4(N(1,1))_k,
\end{eqnarray*}
as we stated in part \ref{partb} of Theorem \ref{thm:capacities}.

\section{Integrable systems and symplectomorphisms}
In this section, we prove Theorem \ref{thm:symplecto}.
\begin{proof}[Proof of Theorem \ref{thm:symplecto}]
It is clearly enough to prove the statements for $q=(0,0,1)$ and $\Sigma=\{(x_1,x_2,x_3)\in S^2\mid x_3 <0\}$.
Let $\phi:S^2\setminus\{q\}\to\R^2$ denote the stereographic projection. Then $\phi$ induces a symplectomorphism $\bar{\phi}:T^*S^2\to T^*\R^2$ by $\bar{\phi}(\xx,\yy)=(\phi(\xx),(\phi^{-1})^*(\yy))$. It follows from a simple calculation that
\[\begin{aligned}
W_\infty&:=\bar{\phi}(D^*(S^2\setminus\{q\})=\left\{(\xx,\yy)\in T^*\R^2\mid |\yy|^2<\frac{4}{(1+|\xx|^2)^2}\right\},\\
W_1&:=\bar{\phi}(D^*\Sigma)=\left\{(\xx,\yy)\in T^*\R^2\mid|\xx|<1\text{ and } |\yy|^2<\frac{4}{(1+|\xx|^2)^2}\right\}.\end{aligned}\]
We will show that $W_\infty$ and $W_1$ are symplectomorphic to $P(2\pi,2\pi)$ and $B(2\pi)$, respectively. We will use techniques from integrable systems, similarly to \cite{ramossepe,or_lpsum,bidisk}.

Fix $\epsilon,C>0$ such that $\epsilon<1$ and $C\ge 1$. For $|\xx|^2<C$,
\[H^{\epsilon,C}(\xx,\yy)=\frac{|\yy|^2(1+|\xx|^2)^2}{4}+\frac{\epsilon}{C-|\xx|^2},\qquad J(\xx,\yy)=\xx\times\yy.\]
Here $\xx\times\yy$ denotes the usual angular momentum, i.e., $\xx\times\yy=x_1y_2-x_2y_1$. Now we define a domain
\begin{equation}\label{eq:xec}W^{\epsilon,C}=\{(\xx,\yy)\in T^*\R^2\mid |\xx|^2<C\text{ and } H^{\epsilon,C}(\xx,\yy)\le 1\}.\end{equation} Let $F^{\epsilon,C}:=(H^{\epsilon,C},J):X^{\epsilon,C}\to \R^2$ and let
\begin{equation}\label{eq:hmin}h_{min}^{\epsilon,C}(j)=\min\left\{\frac{j^2(1+u)^2}{4u}+\frac{\epsilon}{C-u}\mid u\in(0,C)\right\}.\end{equation}
It follows from a simple calculation that $h_{min}^{\epsilon,C}:\R\to\R$ is a well-defined smooth function.

We now claim the following facts about $F^{\epsilon,C}$:
\begin{enumerate}[label=(\arabic*)]
\item $H^{\epsilon,C}$ and $J$ Poisson commute, i.e., $\{H^{\epsilon,C},J\}=0$,
\item $F^{\epsilon,C}(W^{\epsilon,C})$ consists of the points $(h,j)\in\R^2$ such that $h_{min}^{\epsilon,C}(j)\le h\le 1$, and $h_{min}^{\epsilon,C}(j)= h$ if, and only if, $(h,j)$ is a critical value of $F^{\epsilon,C}$.
\item $(F^{\epsilon,C})^{-1}(h,j)$ is compact and connected for every $(h,j)\in\R^2$.
\end{enumerate}

\begin{proof}[Proof of the Claims]
Assuming $\xx\neq 0$, we let $(r,\theta)$ denote $\xx$ in polar coordinates and let $(p_r,p_\theta)$ be the induced coordinates for $\yy$. It is easy to see that
\[|\xx|=r,\qquad|\yy|^2=p_r^2+\frac{p_\theta^2}{r^2},\qquad\xx\times\yy=p_\theta.\]
So 
\begin{equation}\label{eq:hj_polar}
F^{\epsilon,C}(\xx,\yy)=\left(\frac{1}{4}\left(p_r^2+\frac{p_\theta^2}{r^2}\right)(1+r^2)^2+\frac{\epsilon}{C-r^2},p_\theta\right).\end{equation}
Therefore $\{H^{\epsilon,C},J\}=0$ for $\xx\neq 0$. By continuity $\{H^{\epsilon,C},J\}=0$ for all $(\xx,\yy)\in W^{\epsilon,C}$, proving (1).

Now suppose that $F^{\epsilon,C}(\xx,\yy)=(h,j)$ for some $(\xx,\yy)\in W^{\epsilon,C}$. It follows from \eqref{eq:xec}, \eqref{eq:hmin} and \eqref{eq:hj_polar} that
\begin{equation}\label{eq:hj}1\ge h\ge\frac{j^2}{4r^2}(1+r^2)^2+\frac{\epsilon}{C-r^2}\ge h_{min}^{\epsilon,C}(C).\end{equation}
Conversely, if $h_{min}^{\epsilon,C}(j)\le h\le 1$, then there exists $r_0\in(0,\sqrt{C})$ such that \begin{equation}\label{eq:r0}\frac{j^2(1+r_0^2)^2}{4r_0^2}+\frac{\epsilon}{C-r_0^2}=h.\end{equation}
Let $\xx=(r_0,0)$ and $\yy=(0,j/r_0)$. It follows from a simple calculation that $F^{\epsilon,C}(\xx,\yy)=(h,j)$. To determine the critical set of $F^{\epsilon,C}$, we compute $dH_{(\xx,\yy)}$ and $dJ_{(\xx,\yy)}$ in polar coordinates.
\begin{equation}\label{eq:hj2}
\begin{aligned}
dH^{\epsilon,C}_{(\xx,\yy)}&=\frac{p_r(1+r^2)^2}{2}dp_r+\left(\frac{d}{dr}\left(\frac{j^2(1+r^2)^2}{4r^2}+\frac{\epsilon}{C-r^2}\right)+p_r^2r\right)dr\\&+\frac{j(1+r^2)^2}{2r^2}dp_\theta,\\
dJ_{(\xx,\yy)}&=dp_\theta.
\end{aligned}
\end{equation}
Let $(\xx,\yy)$ be a critical point of $F^{\epsilon,C}$ and let $(h,j)=F^{\epsilon,C}(\xx,\yy)$. It follows from \eqref{eq:hj2} that $p_r=0$ and that $r^2$ is the minimizer in \eqref{eq:hmin}. We conclude that $h=h_{min}^{\epsilon,C}(j)$. Conversely, if $(h,j)=F^{\epsilon,C}(\xx,\yy)$ and $h=h_{min}^{\epsilon,C}(j)$, then it follows from \eqref{eq:hj} that $p_r=0$ and that $r^2$ is the minimizer in \eqref{eq:hmin}. So \eqref{eq:hj2} implies that $dH_{(\xx,\yy)}$ and $dJ_{(\xx,\yy)}$ are proportional. Therefore, $(\xx,\yy)$ is a critical point of $F^{\epsilon,C}$, proving (2).

Since $W^{\epsilon,C}$ is bounded, it is clear that $F^{\epsilon,C}(h,j)$ is compact. Now to show it is connected, let $(\xx,\yy)\in F^{\epsilon,C}(h,j)$. Let $r_{min}^{\epsilon,C}(h,j)\le r_{max}^{\epsilon,C}(h,j)$ be the solutions\footnote{It is a simple calculation to verify that \eqref{eq:r0} has exactly two solutions when $h_{min}^{\epsilon,C}(j)<h$ and one solution when $h_{min}^{\epsilon,C}=h$.} of \eqref{eq:r0} in $(0,C)$. It follows from \eqref{eq:hj2} that flowing $(\xx,\yy)$ along the Hamiltonian vector field $X_{H^{\epsilon,C}}$ the $r$-coordinate will eventually reach a point $(\widetilde{\xx},\widetilde{\yy})$ such that $|\widetilde{\xx}|=r_{max}^{\epsilon,C}(h,j)$ and whose $p_r$-coordinate is $0$. From a straight-forward calculation we obtain that \[\widetilde{\yy}=\frac{j}{r_{max}^{\epsilon,C}(h,j)^2}i\widetilde{\xx}.\]
After applying the flow of $X_J$ for time $t=-\arg(\xx)$, we reach the point \begin{equation}\label{eq:p0}p_0:=\left(\left(r_{max}^{\epsilon,C}(h,j),0\right),\left(0,\frac{j}{r_{max}^{\epsilon,C}(h,j)}\right)\right).\end{equation} Since $(\xx,\yy)$ was arbitray, we can connect $p_0$ to any point in $(F^{\epsilon,C})^{-1}(h,j)$, proving (3).
\end{proof}

Now let \[\widehat{W}^{\epsilon,C}=\{(\xx,\yy)\in X^{\epsilon,C}\mid 1>H^{\epsilon,C}(\xx,\yy)>h_{min}^{\epsilon,C}(J(\xx,\yy))\}.\]
As in \cite{ramossepe,or_lpsum}, and we now apply the classical Arnold-Liouville theorem to conclude that $\widehat{W}^{\epsilon,C}$ admits a Hamiltonian toric action. In other words there exist an open set $\widehat{\Omega}^{\epsilon,C}\subset\R^2$, a symplectomorphism $\Phi^{\epsilon,C}$ and a diffeomorphism $\varphi^{\epsilon,C}$ such that the following diagram commutes.
\[
\xymatrix{
 \quad \widehat{W}^{\epsilon,C}\quad\ar[d]^{F^{\epsilon,C}} \ar[r]^{\Phi^{\epsilon,C}} & \;\widehat{\Omega}^{\epsilon,C} \times\mathbb{T}^n \ar[d]^{\text{pr}_1}\; \\
F^{\epsilon,C}(\widehat{W}^{\epsilon,C})\ar[r]^{\quad\varphi^{\epsilon,C}}          & \quad \widehat{\Omega}^{\epsilon,C} \quad }
\]
Moreover we can compute $\varphi^{\epsilon,C}$ as follows. We fix two smooth families of simple curves $\{\gamma_1^{(h,j)},\gamma_2^{(h,j)}\}$ in $(F^{\epsilon,C})^{-1}(h,j)$ generating its first homology and a primitive $\lambda$ of $\omega_0$. So we can choose $\varphi^{\epsilon,C}$ to be the function
\[\varphi^{\epsilon,C}(h,j)=\left(\int_{\gamma_1^{(h,j)}}\lambda,\int_{\gamma_2^{(h,j)}}\lambda\right).\]
In our case, we take $\lambda$ to be the standard Liouville form in $T^*\R^2$ and we define $\gamma_1^{(h,j)}$ and $\gamma_2^{(h,j)}$ as follows. For $(h,j)$, let $r_{max}^{\epsilon,C}(h,j)$ be the largest solution of \eqref{eq:r0} as above. Let $\sigma_0$ be the curve following the flow of $X_{H^{\epsilon,C}}$ from the point $p_0$ defined in \eqref{eq:p0} until the next point $(\widetilde{\xx},\widetilde{\yy})$ such that $|\widetilde{\xx}|=r_{max}^{\epsilon,C}(h,j)$. As in the proof of (3) above, $(\widetilde{\xx},\widetilde{\yy})$ can be joined to $p_0$ by the flow of $X_J$. There are two simple ways to do that. We define $\sigma_1$ and $\sigma_2$ to be the curves from $(\widetilde{\xx},\widetilde{\yy})$ to $p_0$ obtained by flowing by $X_J$ by time $-\arg(\widetilde{\xx})\in(-2\pi,0)$ and $2\pi-\arg(\widetilde{\xx})\in(0,2\pi)$, respectively. Finally, for $i=1,2$, we let $\gamma_i^{(h,j)}$ be a smoothening of the composition of $\sigma_0$ with $\sigma_i$. For $j\neq 0$, we can compute $\varphi^{\epsilon,C}(h,j)$ using polar coordinates as follows:
\begin{equation}\label{eq:int}
\int_{\gamma_i^{(h,j)}}\lambda=\int_{\gamma_i^{(h,j)}}p_r\,dr+\int_{\gamma_i^{(h,j)}}p_\theta\,d\theta=\int_{\sigma_0}p_r\,dr+\int_{\gamma_i^{(h,j)}}p_\theta\,d\theta.\end{equation}
We observe that
\begin{equation}\label{eq:int2}
\begin{aligned}
\int_{\gamma_i^{(h,j)}}p_\theta\,d\theta&=\left\{\begin{aligned}2\pi j,&\text{ if }i=2,j>0,\\
0,&\text{ if }i=2,j<0 \text{ or }i=1,j>0,\\
-2\pi j&\text{ if }i=1, j<0.\end{aligned}\right.=:\Theta_i(j)\\
\int_{\sigma_0}p_r\,dr&=2\int_{r_{min}^{\epsilon,C}(h,j)}^{r_{max}^{\epsilon,C}(h,j)}\sqrt{\frac{4}{(1+r^2)^2}\left(h-\frac{\epsilon}{C-r^2}\right)-\frac{j^2}{r^2}}\,dr.
\end{aligned}
\end{equation}
As in the commutative diagram above, we let $\widehat{\Omega}^{\epsilon,C}=\varphi^{\epsilon,C}(F^{\epsilon,C}(\widehat{W}^{\epsilon,C}))$. It is simple to check that $\widehat{\Omega}^{\epsilon,C}\times \mathbb{T}^n$ is symplectomorphic to 
\[X_{\widehat{\Omega}^{\epsilon,C}}=\left\{(z_1,z_2)\in\mathbb{C}^2\mid (\pi|z_1|^2,\pi|z_2|^2)\in \widehat{\Omega}^{\epsilon,C}\right\}.\]
Let $\Omega^{\epsilon,C}$ be the closure of $\widehat{\Omega}^{\epsilon,C}$ in $\R^2_{\ge 0}$. Similarly to \cite{bidisk} and \cite{or_lpsum}, we can use theorems of Eliasson \cite{eliasson} to extend $\Phi^{\epsilon,C}$ to a symplectomorphism $W^{\epsilon,C}\cong X_{\Omega^{\epsilon,C}}$.

We now observe that 
\begin{equation*}
\begin{aligned}
W_\infty=\bigcup_{\epsilon>0} W^{\epsilon,1/\sqrt{\epsilon}},\qquad
W_1=\bigcup_{\epsilon>0}W^{\epsilon,1}.
\end{aligned}
\end{equation*}
It is a simple calculus exercise to verify that 
\[\epsilon_1<\epsilon_2\Rightarrow \Omega^{\epsilon_2,1/\sqrt{\epsilon_2}}\subset \Omega^{\epsilon_1,1/\sqrt{\epsilon_1}} \text{ and }\Omega^{\epsilon_2,1}\subset\Omega^{\epsilon_1,1}.\] Let
\begin{equation}\label{eq:omega}
\Omega_\infty=\bigcup_{\epsilon>0} \Omega^{\epsilon,1/\sqrt{\epsilon}},\qquad
\Omega_1=\bigcup_{\epsilon>0}\Omega^{\epsilon,1}.
\end{equation}
Arguing as in \cite{or_lpsum}, we obtain symplectomorphisms
\begin{equation}\label{eq:w}
\begin{aligned}
W_\infty&\cong\bigcup_{\epsilon>0} X_{\Omega^{\epsilon,1/\sqrt{\epsilon}}}=X_{\Omega_\infty},\\
W_1&\cong\bigcup_{\epsilon>0}X_{\Omega^{\epsilon,1}}=X_{\Omega_1}.
\end{aligned}
\end{equation}
We note that $\Omega_\infty$ and $\Omega_1$ are relatively open in $\R^2_{\ge 0}$. To finish the proof of the proposition, we need to compute $\Omega_\infty$ and $\Omega_1$. 

The boundary of $\Omega^{\epsilon,1/\sqrt{\epsilon}}$ is a curve parametrized by $\varphi^{\epsilon,1/\sqrt{\epsilon}}(1,j)$. So $\Omega_\infty$ is the relatively open set bounded by the curve \[\left(\rho_1(j),\rho_2(j)\right)=\lim_{\epsilon\to 0}\varphi^{\epsilon,1/\sqrt{\epsilon}}(1,j).\] The domain of the parametrization is $[-1,1]$ as $\epsilon\to 0$.  Since this curve is continuous, it suffices to compute $\rho_1(j)$ and $\rho_2(j)$ for $0<j^2<1$.
Let $r(j)\le R(j)$ be the positive roots of $4r^2-j^2(1+r^2)^2=0$. It follows from \eqref{eq:int} and \eqref{eq:int2} that \begin{equation}\label{eq:rho}\rho_i(j)=2\int_{r(j)}^{R(j)}\sqrt{\frac{4}{(1+r^2)^2}-\frac{j^2}{r^2}}\,dr+\Theta_i(j).\end{equation}
We now compute the integral above.
\begin{equation}\label{eq:int3}\begin{aligned}
&2\int_{r(j)}^{R(j)}\sqrt{\frac{4}{(1+r^2)^2}-\frac{j^2}{r^2}}\,dr\\&=2\int_{r(j)}^{R(j)}\frac{4 r^2-j^2(1+r^2)^2}{r(1+r^2)\sqrt{4 r^2-j^2(1+r^2)^2}}\,dr\\
&=\Bigg[2\arcsin\left(\frac{r^2-1}{\sqrt{1-j^2}(r^2+1)}\right)-|j|\arcsin\left(\frac{j^2r^2+j^2-2}{2\sqrt{1-j^2}}\right)\\&\quad+|j|\arcsin\left(\frac{j^2+j^2r^2-2r^2}{2\sqrt{1-j^2}r^2}\right)\Bigg]_{r=r(j)}^{r=R(j)}\\
&=2\left(\arcsin 1-\arcsin(-1)\right)-|j|(\arcsin 1-\arcsin(-1))+|j|(\arcsin(-1)-\arcsin 1)\\&=2\pi-2\pi|j|.
\end{aligned}\end{equation}
Combining \eqref{eq:rho} and \eqref{eq:int3} we obtain
\[\left(\rho_1(j),\rho_2(j)\right)=\left\{\begin{aligned}(2\pi(1-j),2\pi),\quad&\text{if }0<j<1,\\
(2\pi,2\pi(1+j)),\quad&\text{if }-1<j<0.\end{aligned}\right.\]
So $\Omega_\infty=[0,2\pi)\times[0,2\pi)$. Therefore $(D^*(S^2\setminus\{q\})$ is symplectomorphic to $P(2\pi,2\pi)$.

Arguing similarly, $\Omega_1$ is the relatively open set in $\R^2_{\ge 0}$ bounded by the coordinate axes and the curve
\[(\tilde{\rho}_1(j),\tilde{\rho}_2(j))=\lim_{\epsilon\to 0} \varphi^{\epsilon,1}(1,j).\] Again the domain of the parametrization is $[-1,1]$. We now compute $\tilde{\rho}_i(j)$ for $0<|j|<1$.
\[\begin{aligned}\tilde{\rho}_i(j)&=2\int_{r(j)}^{1}\sqrt{\frac{4}{(1+r^2)^2}-\frac{j^2}{r^2}}\,dr+\Theta_i(j)\\&=\Bigg[2\arcsin\left(\frac{r^2-1}{\sqrt{1-j^2}(r^2+1)}\right)-|j|\arcsin\left(\frac{j^2r^2+j^2-2}{2\sqrt{1-j^2}}\right)\\&\quad+|j|\arcsin\left(\frac{j^2+j^2r^2-2r^2}{2\sqrt{1-j^2}r^2}\right)\Bigg]_{r=r(j)}^{r=1}+\Theta_i(j)\\&=\pi-\pi|j|+\Theta_i(j)\\&=\pi+(-1)^i\pi j.\end{aligned}\]
So \[(\widetilde{\rho}_1(j),\widetilde{\rho}_2(j))=(\pi(1-j),\pi(1+j)),\quad \text{for }0<|j|<1.\]
By continuity this expression holds for all $j\in[-1,1]$. Therefore $D^*\Sigma$ is symplectomorphic to $B(2\pi)$.
\end{proof}

\bibliographystyle{alpha}
\bibliography{biblio}
\end{document}